\documentclass{amsart}

\usepackage[a4paper]{geometry}
\usepackage{amsmath,amsthm,amsfonts,amssymb,enumerate,amsbsy}
\usepackage{mathrsfs}
\usepackage{mathtools}
\usepackage{tikz,tikzscale}
\usepackage{graphicx,subcaption,caption}
\usetikzlibrary{decorations.markings}
\usetikzlibrary{graphs,graphs.standard,arrows.meta}
\usetikzlibrary{arrows,shapes,cd}
\usetikzlibrary{positioning}
\usepackage[nodisplayskipstretch]{setspace}
\usepackage{environ}
\usepackage{pgfplots,pgfplotstable,pgf}
\usepgfplotslibrary{groupplots}
\pgfplotsset{compat=newest, ticks=none}
\usepackage[final]{pdfpages}
\usepackage{import}
\usepackage{xifthen}
\usepackage{transparent}
\usepackage{enumerate}
\usepackage{xparse}

\makeatletter
\newsavebox{\measure@tikzpicture}
\NewEnviron{scaletikzpicturetowidth}[1]{%
	\def\tikz@width{#1}%
	\def\tikzscale{1}\begin{lrbox}{\measure@tikzpicture}%
		\BODY
	\end{lrbox}%
	\pgfmathparse{#1/\wd\measure@tikzpicture}%
	\edef\tikzscale{\pgfmathresult}%
	\BODY
}
\makeatother
\tikzset{->-/.style={decoration={
			markings,
			mark=at position #1 with {\arrow{>}}},postaction={decorate}}}
\tikzset{cross/.style={cross out, draw=black, minimum size=2*(#1-\pgflinewidth), inner sep=0pt, outer sep=0pt},
	cross/.default={3pt}}

\NewDocumentCommand\set{mg}{%
	\ensuremath{\bigl\{ #1 \IfNoValueTF{#2}{}{\bigm| #2} \bigr\}}%
}

\newcommand{\bigzero}{\mbox{\normalfont\Large\bfseries 0}}

\newcommand{\R}{\mathbb{R}}
\newcommand{\RP}{\mathbb{RP}}
\newcommand{\C}{\mathbb{C}}

\newcommand{\N}{\mathbb{N}}

\newcommand{\pib}{\overline{\pi}}
\newcommand{\NHIM}{\mathcal{N}}

\newcommand{\e}{\varepsilon}
\renewcommand{\a}{\alpha}
\renewcommand{\b}{\beta}
\renewcommand{\d}{\delta}
\newcommand{\g}{\gamma}

\newcommand{\ra}{r_\a}
\newcommand{\ba}{\b_\a}
\newcommand{\da}{\d_\a}
\newcommand{\ga}{\g_\a}
\newcommand{\tga}{{\tilde{\g}_\a}}

\newcommand{\ag}{{\a_\g}}
\newcommand{\bg}{{\b_\g}}
\newcommand{\dg}{{\d_\g}}
\newcommand{\rg}{{r_\g}}
\newcommand{\trg}{{\tilde{r}_\g}}

\DeclareMathOperator{\Ima}{Im}
\DeclareMathOperator{\re}{Re}
\DeclareMathOperator{\im}{Im}
\DeclareMathOperator{\Span}{Span}

\newcommand{\sbc}{\mathscr{C}}

\numberwithin{equation}{section}
\numberwithin{figure}{section}

\theoremstyle{plain}
\newtheorem{thm}{Theorem}[section]
\newtheorem{proposition}[thm]{Proposition}
\newtheorem{corollary}[thm]{Corollary}
\newtheorem{lemma}[thm]{Lemma}

\theoremstyle{definition}
\newtheorem{definition}[thm]{Definition}

\newtheorem{remark}[thm]{Remark}

\patchcmd{\subsubsection}{-.5em}{.5em}{}{}
\patchcmd{\subsection}{-.5em}{.5em}{}{}

\title[Regularisation for Simultaneous Binary Collisions]{On the $ C^{8/3} $-Regularisation of Simultaneous Binary Collisions in the Planar 4-Body Problem}

\begin{document}
	\author{Nathan Duignan}
	\address[Nathan Duignan]{Department of Applied Mathematics, University of Colorado, Boulder, CO, 80309-0526 USA}
	
	\author{Holger R.~Dullin}
	\address[Holger R.~Dullin]{School of Mathematics and Statistics, University of Sydney, Camperdown, 2006 NSW, Australia}
	\date{}
	\begin{abstract}
    %   In the 4-body problem two binary collisions may occur simultaneously. 
    %   Mart\'inez and Sim\'o \cite{Martinez2000} proved that in the \emph{collinear} 4-body problem this simultaneous binary collision (SBC) can be block-regularised, but that the block map is not smooth but only $C^{8/3}$ differentiable.
    %   In this paper we prove that in the \emph{planar} 4-body problem the SBC is only $C^{8/3}$ regularisable as well.
    %   This is achieved through a geometric proof that uses blow-up, normal form, dynamics near normally hyperbolic manifolds of equilibrium points and Dulac maps.
      The dynamics of the 4-body problem allows for two binary collisions to occur simultaneously. It is known that in the \emph{collinear} 4-body problem this simultaneous binary collision (SBC) can be block-regularised, but that the resulting block map is only $C^{8/3}$ differentiable. In this paper, it is proved that the $C^{8/3}$ differentiability persists for the SBC in the \emph{planar} 4-body problem. The proof uses several geometric tools, namely, blow-up, normal forms, dynamics near normally hyperbolic manifolds of equilibrium points, and Dulac maps.
	\end{abstract}
	\maketitle

\section{Introduction}

The $N$-body problem has always driven progress in dynamical systems. This is
spectacularly evidenced by Poincar\'e's discovery of chaos in the 3-body problem,
see, e.g., Chenciner's recent review \cite{Chenciner15}. The extraordinary richness of 
the dynamics of the 3-body problem, as proved by Moeckel and Montgomery 
\cite{MoeckelMontgomery15}, can be traced in part to the 
dynamics near the triple collision. While the binary collision
has been known classically \cite{LeviCivita20} to be regularisable, 
McGehee showed  \cite{McGehee1974} that
the triple collision cannot be regularised. 
Consequently, it serves as a building block of complicated dynamics.
A new type of collision appears in the 4-body problem, where 
two binary collisions can occur in separate locations in space,
but at the same time. The regularisability of this simultaneous 
(or double) binary collision (SBC) in the planar 4-body problem
is the subject of this paper. It has been conjectured by Sim\'o and Mart\'inez \cite{Martinez1999}
that the SBC can be regularised, but, only in such a way that the regularisation
is finitely smooth, precisely $C^{8/3}$.

Initial investigations by Sim\'o and Lacomba \cite{Simo1992}, as well as Belbruno \cite{belbrunoSimultaneousDoubleCollision1984},
revealed that the SBC is topologically regularisable.
% , but withheld from investigating higher differentiability.
%
In \cite{Elbialy1993planar,Simo1992,Martinez1999} the SBC was proved to be 
$C^1$ regularisable. 
Based on numerical evidence Sim\'o and Mart\'inez 
\cite{Martinez1999}
conjectured the regularisation was only $C^{8/3}$
and, in a paper by the same authors \cite{Martinez2000}, they  
proved that this $C^{8/3}$ differentiability holds in the \emph{collinear} 4-body problem. 
More precisely, the finite differentiability means that an isolating block, as introduced by Conley and Easton \cite{Conley1971,Easton1971}, 
can be constructed around the SBC and the consequent block map can be extended to a map 
that is at most, and at least, $C^{8/3}$ differentiable at orbits going to collision. For further context on the problem and details on relevant theory see \cite{duignanRegularisationSimultaneousBinary2019a}.

The primary achievement of this paper is to extend the $C^{8/3}$ result to the \emph{planar} 4-body problem
in Theorem~\ref{thm:C2Regularisation}. Whilst it is immediate from the collinear result that collinear collision orbits in the
planar problem will exhibit this finite differentiabilty, we show that the block map at non-collinear collision orbits 
 is still at least $C^{8/3}$.

Before describing in more detail the strategy of the proof we would like 
to explain the essence of the mechanism in a toy model first
introduced in \cite{duignanRegularisationPlanarVector2019}.
How can non-smoothness arise in a smooth (or even polynomial) planar dynamical system?
The answer to this question as described in \cite{duignanRegularisationPlanarVector2019}
is: Through a block map past a degenerate equilibrium point. 
A degenerate equilibrium point can, e.g., be such that the negative $x$-axis is the stable
manifold and the positive $x$-axis is the unstable manifold. 
In this case, a Poincar\'e map between transversal sections to the $x$-axis
can be defined. In \cite{duignanRegularisationPlanarVector2019} it was shown 
that this map generically cannot be extended to $y=0$ in a smooth way.
The reason is that orbits passing the equilibrium with $y > 0$ may 
be sufficiently different from those passing the equilibrium with $y < 0$.
This difference can be precisely characterised after replacing the degenerate point with a copy of $S^1$; a so called \emph{blow-up}.
In the collinear problem the fixed point is replaced
with a manifold of fixed points, which are in turn replaced by copies of $S^1$ \cite{duignanC83regularisationSimultaneous2020}.
The story is similar in the planar problem, however, the dimensionality of the problem requires the
blow up to replace points by a 3-dimensional manifold  instead of $S^1$.

A precursor of this work is \cite{duignanC83regularisationSimultaneous2020}, where the SBC in the 
collinear 4-body problem is treated. A key difference to the first proof given in \cite{Martinez2000} is that
the proof of \cite{duignanC83regularisationSimultaneous2020} is more geometric and allows for a generalisation to the planar problem. The idea
behind this previous work and the proof for the planar problem is the same.
After regularising the simple binary collisions, scaling time, and 
using the (approximately constant) energy of the binaries as a coordinate, the consequent vector field of the $4$-body problem
has a manifold of degenerate equilibria corresponding to SBC for which the linear
part of the vector field vanishes. A normal form procedure and a blow-up allows the construction of
an asymptotic expansion of the block map.
However, due to the higher codimension of the manifold of equilibria in the planar compared to the collinear problem, some nontrivial extensions of the 
normal form theory, blow-up procedure, and asymptotics of the transition map have been made.

The paper is structured as follows. In Section~\ref{sec:coordinates} we introduce \textit{generalised Levi-Civita} coordinates which include the 
energy of the binaries and two complex coordinates $\zeta_1, \zeta_2$ whose imaginary part is the
(approximately constant) angular momentum of the binaries. 
This generalises work of Elbialy \cite{Elbialy1993collinear} on the collinear problem. 

To resolve the dynamics near the degenerate equilibria $\sbc$ (the set of collision points) 
a blow-up is performed in Section~\ref{sec:C0regularity}, 
which augments $\sbc$ to obtain a smooth collision manifold $ \mathcal{C} $.
In prior proofs of the $ C^0 $-regularity, McGehee type coordinates were used to blow-up the singularity \cite{Elbialy1993planar,Martinez1999}. 
Essentially, in these coordinates  $ \sbc $ is augmented with copies of $ S^3 $ to produce the collision manifold. We will instead use a blow-up procedure which augments $ \sbc $ with copies of $ \RP^3 $. This will be shown to have several benefits; the regularisation of isolated binary collisions, a relatively simple argument proving the union of the collision and ejection orbits forms a smooth manifold, and that the flow on the collision manifold is integrable. Ultimately, this blow-up procedure shows that the collision manifold $ \mathcal{C} $ is foliated by the homoclinic connections of a manifold of normally hyperbolic saddle singularities. A new proof of the $ C^0 $-regularity follows.

To simplify the analysis in Section~\ref{sec:Ckregularisation}, a nonlinear formal normal form is computed
at an arbitrary degenerate equilibrium point. Since the equilibrium points are degenerate,
a non-standard type of normal form due to \cite{stolovitchProgressNormalForm2009}
is used. The homological operator is not an automorphism in this case
as the leading order terms are quadratic instead of linear. Moreover, we use a restriction of the 
quadratic part to the collinear problem to construct the homological operator which ultimately produces approximate integrals near SBC.
The final step is the computation of the block map through the composition of three simpler
maps obtained in the blown up system. The computation of these so called Dulac maps is based
on work of \cite{roussarieBifurcationPlanarVector1998,dumortierSmoothNormalLinearization2010,mourtada1990cyclicite} and its recent extension 
from fixed points in the plane to manifolds of fixed points \cite{duignanNormalFormsManifolds} in $\R^n$.

The  road-map just described provides insight into why the 
smoothness is a peculiar $8/3$. The blow up generates hyperbolic 
equilibria that have resonant eigenvalues 1 and 3. The corresponding Dulac
maps have components whose leading order consists of powers $1/3$ and $3$.
If the transition map describing the dynamics along the homoclinic connection 
of the hyperbolic point would be the identity map these powers would cancel and 
the composed map would be $C^\infty$ smooth. Instead, the first terms in the normal form that cannot be removed appear at order 9, and this results in a term of order 8 in the transition map.
The composition of these maps hence gives a map that is only $C^{8/3}$.
The appearance of the resonant terms in normal form can also be interpreted as
the absence of an invariant foliation of a neighbourhood of the manifold of equilibria.

% comment on MandM global reg, Knauf?

%\comment{ Now I am really confused, when writing the introduction something that has bothered me all along came up again:
%In the planar problem you first blow up, and then do the normal form, and then do the same blow up again. 
%I think the reason was that you wanted to proof $C^0$ first, and that $L_i$ were not constant originally.
%In the collinear case you first to normal form and then blow up. Now it seems to me that if you would do 
%the same in the planar problem then you would end up only blowing up the two real coordinates $I_1, I_2$, 
%and so everything would be much more similar to the collinear problem. In your thesis the situation was different, since the %normal form did not produce $L_i' = 0$, and so they had to be included in the blow up. But now we do get $L_i' = 0$ to hight %order, so we could just treat these as part of the manifold of fixed points? This would turn the $L_i$ into constant initial %conditions $L_i^*$. I think there must be something wrong with what I'm saying, but I can't see what at the moment, can you?}

	\section{Coordinates Near Simultaneous Binary Collision}\label{sec:coordinates}
	
	\subsection{The planar 4-body problem}
		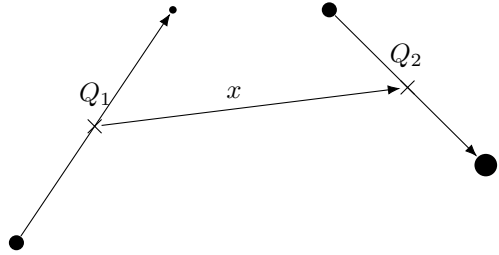
\begin{figure}[ht]
			\centering
			\begin{scaletikzpicturetowidth}{\textwidth}
				\begin{tikzpicture}[scale = \tikzscale]
				\node[circle,fill,inner sep=2pt] (1) at (-2,-0.5) {};
				\node[circle,fill,inner sep=1pt] (2) at (-1,1) {};
				\node[circle,fill,inner sep=2pt] (3) at (0,1) {};
				\node[circle,fill,inner sep=3pt] (4) at (1,0) {};
				\node[cross] (5) at (-1.5,0.25) {};
				\node[cross] (6) at (0.5,0.5) {}; 
				\node[inner sep=2pt,fill=white] (7) at (-4,0.5) {};
				\node[inner sep=2pt,fill=white] (8) at (3,0.5) {};
				
				\draw[-Latex] (1) -- node[label=north:$Q_1$] {}(2);
				\draw[-Latex] (3) -- node[label=north:$Q_2$] {}(4);
				\draw[-Latex] (5) -- node[auto] {$x$}(6);
				\end{tikzpicture}
			\end{scaletikzpicturetowidth}
			\caption{The configuration variables near simultaneous binary
        collision}\label{fig:4BPDiff}
		\end{figure}
		
		Suppose there are four bodies in the plane consisting of two binaries
    undergoing collision in different regions of configuration space at
    precisely the same time $ t_c $. Further, suppose that the bodies with mass $
    m_1 $ and $ m_2 $ undergo one of the binary collisions and bodies with
    masses $ m_3 $ and $ m_4 $ undergo the other. We will call these two
    binaries the \textit{distressed binaries}. Let the difference vector between
    the bodies in each binary be given by $ Q_1,Q_2 \in \C $ respectively and
    let $ x\in\C $ be the difference vector between the two centre of masses of
    the binaries. The coordinates are depicted in Figure \ref{fig:4BPDiff}. If $
    P_1,P_2,y\in\C^* $ are the conjugate momenta of $ Q_1,Q_2,x $, the dynamics
    is given by the Hamiltonian,
		\begin{equation}
			\begin{aligned}
				H(Q,x,P,y) = \sum_{j=1}^{2}\left( \frac{1}{2 M_j} |P_j|^2 - k_j
          |Q_j|^{-1} \right) + \frac{1}{2}\mu |y|^2 - \hat{K}(Q_1,Q_2,x),
			\end{aligned}		
		\end{equation}
		with symplectic form $ \omega = d\Theta $, and $ \Theta $ is the
    tautological one-form on $ T^*\C^3 $, $ \Theta = \re\left(\bar{P}_1 dQ_1 +
      \bar{P}_2 dQ_2 + \bar{y}dx \right) $. The function $ \hat{K} $ contains the potential
    terms coupling the two binaries and is smooth in a neighbourhood of
    simultaneous binary collision which occurs at $Q_1=Q_2=0$. Each of $
    M_j,k_j,\mu,d_j, c_j > 0 $ are constant functions of the masses given
    explicitly:
		\begin{equation}\label{eqn:massconsts}
			\begin{aligned}
        \hat{K} &= \frac{d_1}{|x + c_2 Q_1 - c_4 Q_2|} + \frac{d_2}{|x + c_2 Q_1 + c_3 Q_2|} + \frac{d_3}{|x - c_1 Q_1 - c_4 Q_2|} + \frac{d_4}{|x - c_1 Q_1 + c_3 Q_2|} ,\\
				&~ \\
				M_1 &= \frac{m_1 m_2}{m_1 + m_2},\qquad M_2 = \frac{m_3 m_4}{m_3 + m_4},\qquad k_1 = m_1 m_2,\qquad k_2 = m_3 m_4, \\
				\mu &= \frac{m_1 + m_2 + m_3 + m_4}{(m_1 + m_2)(m_3 + m_4)},\\
				d_1 &= m_1 m_3,\quad  d_2 = m_1 m_4,\quad d_3 = m_2 m_3,\quad d_4 = m_2 m_4, \\
				c_1 &= m_2^{-1} M_1,\quad c_2 = m_1^{-1} M_1,\quad c_3 = m_4^{-1}
        M_2,\quad c_4 = m_3^{-1} M_2.
			\end{aligned}
		\end{equation}
		This choice of coordinates achieves a reduction of the system by translational
    symmetry. The coordinates have been chosen so that the mass metric is diagonal.
	
		It is more convenient to work with rescaled variables $ \tilde{Q}_j, \tilde{P}_j $ via a symplectic transformation
		\begin{equation}
			\tilde{Q}_j = 4 k_j M_j Q_j,\qquad \tilde{P}_j = (4 k_j M_j)^{-1} P_j.
		\end{equation}
		The rescaled Hamiltonian is then
		\begin{equation}
			H(\tilde{Q},x,\tilde{P},y) = \sum_{j=1}^{2} \frac{1}{2} a_j \left( |\tilde{P}_j|^2 - \frac{1}{2} |\tilde{Q}_j|^{-1} \right) + \frac{1}{2}\mu |y|^2 - \tilde{K}(\tilde{Q},x),
		\end{equation}
		where $\displaystyle a_j = 16 k_j^2 M_j $ and $ \tilde{K}(\tilde{Q}_j,x) = \hat{K}\left((4 k_j M_j)^{-1} \tilde{Q}_j,x\right) $.
	
	The system has rotational symmetry and hence the total angular momentum $\im(\bar{Q}_1 P_1) + \im( \bar{Q}_2 P_1 ) + \im(\bar{x} y)$ is conserved.
	Moreover, the usual scaling symmetry of the 4-body problem in these variables reads
	\begin{equation} \label{eqn:scal}
	    Q_i \to s Q_i, \quad x \to s x, \quad P_i \to s^{-1/2} P_i, \quad y \to s^{-1/2} y, \quad H \to H/s \,.
	\end{equation}

	\subsection{Levi-Civita Regularisation of Binary Collisions}\label{sec:LeviCivita}
	
		The next stage of transformations regularise the binary collision of each distressed binary. This is done by passing to the \textit{Levi-Civita variables} via the symplectic map 
		\[ \tilde{Q}_j = \frac{1}{2}\tilde{z}_j^2,\qquad \tilde{P}_j = \bar{\tilde{z}}_j^{-1} u_j \]
		resulting in the partially regularised, translation reduced Hamiltonian,
		
		\begin{equation}\label{eqn:PartiallyRegularisedHam}
			H(\tilde{z},x,u,y) = \sum_{j=1}^{2} \frac{1}{2} a_j |\tilde{z}_j|^{-2}\left(  |u_j|^2 - 1 \right) 
			+ \frac{1}{2}\mu |y|^2 - \tilde{\tilde{K}}(\tilde{z}_1,\tilde{z}_2,x),
		\end{equation}
		with $ \tilde{\tilde{K}}(\tilde{z}_1,\tilde{z}_2,x):= \tilde{K}(\frac{1}{2} \tilde{z}_1^2,\frac{1}{2}\tilde{z}_2^2,x) $.
		
		Time can be rescaled to $ dt = |\tilde{z}_1|^2 |\tilde{z}_2|^2 d\tau $ by using the Poincar\'{e} trick of moving to extended phase space and restricting to a constant energy surface. That is, by constructing the new Hamiltonian 
		\begin{equation} \label{eqn:HamPoincare}
			\begin{aligned}
				\mathscr{H}(\tilde{z},x,u,y) &= |\tilde{z}_1|^2 |\tilde{z}_2|^2\left(H(\tilde{z},x,u,y) - h\right)\\
					&=  \frac{1}{2} a_1 |\tilde{z}_2|^2\left(  |u_1|^2 -  1\right) + \frac{1}{2} a_2 |\tilde{z}_1|^2\left( | u_2|^2 -  1\right) 
					+ |\tilde{z}_1|^2 |\tilde{z}_2|^2\left(\frac{1}{2}\mu |y|^2 - \tilde{\tilde{K}}(\tilde{z},x) - h\right),
			\end{aligned}
		\end{equation}
		and restricting to a constant energy surface in the original Hamiltonian, $ H = h $, yielding $ \mathscr{H} = 0 $. The flow on $ \mathscr{H} = 0 $ is equivalent to the flow on $ H = h $ up to time rescaling. The Hamiltonian $ \mathscr{H} $ is regular at $ \tilde{z}_1 = 0 $ or $ \tilde{z}_2 = 0 $ and so the binary collision singularities have been regularised. However, the set of simultaneous binary collisions $ \tilde{z}_1 = \tilde{z}_2 = 0 $, denoted by $ \sbc $, is a critical point of $ \mathscr{H} $ and hence the associated Hamiltonian vector field has fixed points at $ \sbc $. Essentially, in rescaling time to regularise the binary collisions, the vector field was over scaled near $ \sbc $, slowing down orbits as they approach the singularity and creating a manifold of equilibria. In the new time $ \tau $, collision (resp. ejection) orbits approach $ \sbc $ as $ \tau\to\infty $ (resp. $ \tau\to -\infty $).

		The Hamiltonian \eqref{eqn:HamPoincare} is canonical and regular at collision. 
		To distinguish orbits in the collision limit we need to introduce the intrinsic energies of the distressed binaries as non-canonical coordinates. This is done in the next section. As a result the further analysis will be done directly on the level of the vector field instead of on the level of the Hamiltonian. Before we do so, let us recall a classical result about the collision limit which is best stated in the canonical coordinates of this section.
		
		Let $ \hat{\C} $ be the set of unit complex numbers and denote by $ \Gamma_j = \exp(i \arg(u_j)) \in \hat{\C} $ the angular component of $ u_j $. The following proposition gives the asymptotic behaviour of collision and ejection orbits. Analogous propositions have been proved in many works, for example, \cite{ElBialy1990,Martinez1999}. 	

		\begin{proposition}\label{prop:asymptotics}
			Suppose that $ (\tilde{z}_1(\tau),\tilde{z}_2(\tau),x(\tau),u_1(\tau),u_2(\tau),y(\tau)) $ is a collision (resp. ejection) orbit. Let $ x^*,y^*\in\C $, $ \Gamma_1^*,
			\Gamma_2^* \in \hat{\C} $ and denote by $ |\tilde{z}|^2 = |\tilde{z}_1|^2 + |\tilde{z}_2|^2 $. Then, 
			\[ u_j \to \Gamma_j^*, \quad \frac{\tilde{z}_j}{|\tilde{z}|} \to a_j^{1/3} \Gamma_j^*, \quad x \to x^*, \quad y \to y^*, \]
			as $ \tau \to \infty $ (resp. $ \tau \to -\infty $).  Moreover, the set of orbits asymptotic to any choice of $ x^*,y^*, \Gamma_1^*,\Gamma_2^*  $ forms a real 3 dimensional manifold.
		\end{proposition}
		A geometrical proof can be constructed using the methods of blow-up and desingularisation. In fact, it follows immediately from Proposition \ref{prop:structureOfCollisionManifold} in Section \ref{sec:C0regularity}. 
				
	\subsection{Generalised Levi-Civita Coordinates}
  It has been argued in \cite{duignanC83regularisationSimultaneous2020} that the Levi-Civita coordinates
  are not ideal for analysis of orbits local to a simultaneous binary collision
  in the collinear 4-body problem. The primary issue arises from the
  dimensionality of the asymptotic orbits given in Proposition
  \ref{prop:asymptotics}. In order for $ \sbc $ to be $ C^0 $-regularisable,
  there must exist a map $ \pib $ taking each collision orbit to a unique
  ejection orbit. However, from Proposition \ref{prop:asymptotics}, each point in $
  \sbc $ has a 3 dimensional manifold of asymptotic orbits. Consequently,
  uniqueness can not be upheld. In \cite{duignanC83regularisationSimultaneous2020}, the so called
  \textit{generalised Levi-Civita coordinates} were introduced to resolve this
  problem. In that work, it became clear that taking approximate integrals as
  coordinates was advantageous to understanding the dynamics near
  collision. In particular, for the collinear problem, the intrinsic energy of
  each binary was taken as a coordinate. This can be done analogously for the
  planar problem. Additionally, the intrinsic angular momentum of each
  binary $ \frac{1}{2}\im(\bar{u}_j \tilde{z}_j)$ is an approximate integral,
  thus a good choice of coordinate. 
		
  Introduce the intrinsic energy of each distressed binary 
  \begin{equation}\label{eq:intrinsicEnergies}
    \tilde{h}_j = \frac{1}{2} a_j |\tilde{z}_j|^{-2}(|u_j|^2 - 1),
  \end{equation}
  and the variables $\tilde{\zeta_j} = \bar{u}_j \tilde{z}_j $. Consider the rescaling
  \begin{equation}
    \zeta_j =  a_j^{-1/3} \tilde{\zeta}_j,\qquad h_j = 2 a_j^{-1/3}\tilde{h}_j.
  \end{equation} 
  Then, the \textit{generalised Levi-Civita coordinates} are given
  by $ (\zeta_1,\zeta_2,h_1,h_2,\Gamma_1,\Gamma_2,x,y) \in
  \C^2\times\R^2\times\hat{\C}^2\times \C^2 $. 
  Notice that these are not canonical coordinates.
		
	The Hamiltonian in the generalised Levi-Civita coordinates is,
		\begin{equation}
			H = \frac{1}{2} a_1^{1/3} h_1 + \frac{1}{2} a_2^{1/3} h_2 + \frac{1}{2}\mu |y|^2 - K(\zeta_1,\zeta_2,h_1,h_2,\Gamma_1,\Gamma_2,x),
		\end{equation}
		where $$ K(\zeta_1,\zeta_2,h_1,h_2,\Gamma_1,\Gamma_2,x) :=
    \tilde{\tilde{K}}\left(a_1^{1/3}U_1(\zeta_1,h_1)^{-1}\Gamma_1\zeta_1,\,
      a_2^{1/3}U_2(\zeta_2,h_2)^{-1}\Gamma_2\zeta_2,\, x\right), $$
    and $U_j(\zeta_j,h_j)$ is the unique solution to
    \begin{equation}
      \label{eq:Udefinition}
      h_j = |\zeta_j|^{-2}U_j^2(U_j^2-1),
    \end{equation}
    satisfying $U_j(0,h) = 1$. Note that $U_j = |u_j|$ is the magnitude of the
    canonical momenta so that $u_j = U_j \Gamma_j$. 
    The condition $U_j(0,h) = 1$ is nothing more than
    ensuring $|u_j|\to 1$ as $\tau\to \infty$ for a collision orbit as per
    Proposition \ref{prop:asymptotics}. 
    Defining $U_j$  through \eqref{eq:Udefinition} is only invertible when $ U_j > 1/\sqrt{2} $. We
    restrict to a sufficiently small neighbourhood of the simultaneous binary
    collision $ \zeta_1 = \zeta_2 = 0 $ so this inequality holds.
		
	The pullback of the symplectic form under the transformation to the intrinsic energies is
		\begin{equation}
      \begin{aligned}
        \omega = \re(d\bar y\wedge dx) 
        &+ \sum_{j=1}^2  \frac{-a_j^{1/3}}{2U_j^2(2U_j^2-1)}\left( |\zeta_j|^2\re\left(d\zeta_j\wedge dh_j \right) + i h_j\im\left( \bar\zeta_j \right)d\zeta_j\wedge d\bar\zeta_j\right) \\ 
        &\qquad + \frac{1}{2} a_j^{1/3}\bar\Gamma_j(d\zeta_j-d\bar\zeta_j)\wedge d\Gamma_j  
      \end{aligned}
    \end{equation}
		The associated Poisson bracket is found as
		\begin{equation}
      \begin{aligned}
        \Pi = 4\re\left(\partial_{x} \wedge \partial_{\bar{y}}\right) &+ \sum 4a_j^{-1/3}\frac{U_j^2(2U_j^2-1)}{|\zeta_j|^2}\re\left(
             \partial_{\zeta_j}\wedge\partial_{h_j} \right) \\
           &\qquad +
             2ia_j^{-1/3}\Gamma_j\left( \frac{ h_j }{|\zeta_j|^2} \im\left( \bar\zeta_j
             \right) \partial_{h_j}\wedge\partial_{\Gamma_j} +  \im\left(
               \partial_{\bar\zeta_j} \right)\wedge \partial_{\Gamma_j} \right)
         \end{aligned}
		\end{equation}
		
	Hamilton's equations $ X_H = \Pi(\cdot,H) $ gives the Hamiltonian vector
    field $ X_H $ for the planar 4-body problem in the generalised Levi-Civita
    coordinates. Explicitly, $X_H$ is given by 
		\begin{equation}\label{eqn:generalisedLeviCivita}
			\begin{aligned}[c]
        \dot\zeta_j &= \frac{U_j^2}{|\zeta_j|^2} + 2h_j + 2i a_j^{-1/3}\im\left( \zeta_j\partial_{\zeta_j} K \right) \\
        \dot h_j &= \frac{4 a_j^{-1/3} U_j^4}{|\zeta_j|^2}\re\left( \partial_{\zeta_j}K \right) \\
        \dot \Gamma_j &= \frac{i}{|\zeta_j|^2}\Gamma_j \im\left( h_j\zeta_j - 2a_j^{-1/3}|\zeta_j|^2 \partial_{\zeta_j}K \right) \\
        \dot x &= \mu y \\
        \dot y &= 2  \partial_{\bar x} K
			\end{aligned}	
		\end{equation}
		By scaling $ X_H $ through a space dependent time rescaling $ dt =
    |\zeta_1|^2 |\zeta_2|^2 d\tau $, the single binary collisions at $\zeta_j =
    0$ are regularised. Denote by $ X = |\zeta_1|^2|\zeta_2|^2 X_H $ the
    rescaled vector field and let ${^\prime}$ denote a derivative with respect
    to the fictitious time $\tau$. The rescaled vector field $X$ is regular 
    when $\zeta_j \to 0$. However, time is ``over-scaled'', which means that the whole 
    vector field $X$ vanishes in the collision limit. What is worse, the family 
    of equilibrium points with $\zeta_1 = \zeta_2 = 0$ has vanishing linear part.
    Thus, in the next section, blow-up is used to study the collision limit.

    The following technical lemma is crucial in obtaining the form of $X_H$ given
    in \eqref{eqn:generalisedLeviCivita}.
    \begin{lemma}
      The following relations hold.
      \begin{equation}
        \label{eq:partialEquivalence}
        \begin{aligned}
          \re\left( \zeta_j\partial_{\zeta_j}K \right) &= 0 \\
          \Gamma_j\partial_{z_j} K &= U_j \partial_{\zeta_j} K \\
          \Gamma_j\partial_{\Gamma_j}K &= 2 \zeta_j \partial_{\zeta_j}K
        \end{aligned}
      \end{equation}
    \end{lemma}
    \begin{proof}
      The proof is a computation. One method of attack is to introduce the
      intermediate transformation $(z_j,u_j,x,y) \to (z_j,\hat
      h_j,\hat\Gamma_j,x,y)$. In these coordinates we have that
      $\partial_{\hat\Gamma_j}K = \partial_{\hat h_j}K = 0$. These two conditions in the
      $(\zeta_j,h_j,\Gamma_j,x,y)$ variables gives the first and last relations
      in the lemma. The second relation comes from pulling back the vector
      $\partial_z$ under the coordinate transformation and applying this to $K$. 
    \end{proof}

    \begin{remark}
      All the terms coupling the two distressed binaries from the potential can
      be removed by setting $K=0$. In doing so, the dynamics of two uncoupled
      Kepler problems is recovered. From \eqref{eqn:generalisedLeviCivita}, the
      equations of motion for the uncoupled problem are given by
      \begin{equation}
			  \begin{aligned}[c]
          \dot\zeta_j &= \frac{U_j^2}{|\zeta_j|^2} + 2h_j \\
          \dot h_j &= 0 \\
          \dot \Gamma_j &= \frac{i}{|\zeta_j|^2}\Gamma_j \im\left( h_j\zeta_j \right) \\
          \dot x &= \mu y \\
          \dot y &= 0.
        \end{aligned}	
      \end{equation}
      We recover from this that the (scaled) intrinsic energy $h_j$, angular momentum $
      \im(\zeta_j)$, and momenta $y$ between the two binaries are conserved in this limit.
    \end{remark}

    \begin{remark} \label{rem:cal}
     The collision limit described in Proposition~\ref{prop:asymptotics} expressed in the new variables reads $\zeta_i \to |\tilde z_i|$.
     In particular, $\zeta_1 \to \zeta_2$ in the collision limit and that the $\zeta_i$ are real.
     One way to permanently achieve real $\zeta_i$ is to restrict to the collinear problem in which all imaginary parts vanish. 
     Since $\tilde \zeta_i = 2 \tilde Q_i \overline{{{\tilde{P}}_i}} = Q_i \bar{P_i}$ another possibility to achieve real $\zeta_i$ is to require that $Q_i \perp P_i$.
     This ensures the angular momentum $\im Q_i \bar{P_i}$ vanishes even thought the dynamics is not collinear.
     The latter method occurs in the rectangular 4-body problem.
     A more general invariant sub-problem is the so-called Caledonian problem \cite{StevesRoy98,RoySteves98,Roy00}.
     There $Q_1 = - Q_2$ and $P_1 = -P_2$ and hence $\tilde \zeta_1 = \tilde \zeta_2$. 
     Moreover, the masses are pairwise equal so that the coefficients $a_1 = a_2$ are the same.
    \end{remark}
	
%%% Local Variables:
%%% TeX-master: "SBC_Planar.tex" 
%%% End:
\section{$ C^0 $-regularity of block map}\label{sec:C0regularity}
	
The primary aim of this section is to provide a proof of Theorem
\ref{thm:C0Regularisable} on the $ C^0 $-regularisation of simultaneous binary
collisions. Blow-up and desingularisation on the set of collisions $ \sbc $ is
implemented and a study of the flow on the resultant collision manifold $
\mathcal{C} $ provides the desired proof. In the process, the flow on the
collision manifold is revealed to possess integrals that are related to the
intrinsic angular momentum of each distressed binary, and to the original time 
variable. The blown up collision manifold $\mathcal{C}$ has dimension 11,
and the flow on it is not Hamiltonian. The non-constant dynamics takes place on 
$\RP^3$, and we will show that there are two additional integrals that make
the flow on $\RP^3$ integrable.
		
\subsection{Blow-Up and Desingularisation}\label{sec:CollisionManifold}

The use of blow-up in celestial mechanics was first introduced by McGehee
\cite{McGehee1974} in his study of the triple collision. Later it was
implemented in investigations of the simultaneous binary collision by Elbialy
\cite{ElBialy1990} and Mart\'{i}nez and Sim\'{o} \cite{Martinez1999}. Both of
these papers perform the blow-up by introducing a set of coordinates akin to the
usual McGehee coordinates. Essentially, the McGehee coordinates replace each
simultaneous binary collision with a copy of $ S^3 $. A study of the resulting
flow on each $ S^3 $ provides topological properties of the flow near collision.
However, in both works, the McGehee coordinates are taken before Levi-Civita
regularisation. Consequently, each binary collision is not regularised and one
needs to perform a more trying analysis of the flow on each $ S^3 $ to prove
statements like the $ C^0 $ regularity. Instead of using the McGehee
coordinates, a blow-up procedure after Levi-Civita regularisation will be
implemented. Further, the more algebraic route of blowing up each collision with
a copy of $ \RP^3 $ is followed; see \cite{Ilyashenko2008} for more details. The
advantage of blowing up with $ \RP^3 $ instead of $ S^3 $ is both the avoidance
of trig functions and the relatively simple proof of Corollary
\ref{cor:collisionEjectionIsSmooth} which shows the set of collision and
ejection orbits is a smooth manifold.
		
Before introducing the blow-up for the simultaneous binary collision, it will be
useful to first describe the blow-up procedure for a point in $ \C^2 \cong
\R^4$. Take $ (\zeta_1,\zeta_2)\in\C^2 $, let $ \zeta_1 = \zeta_{11} + i
\zeta_{12}, \zeta_2 = \zeta_{21} + i \zeta_{22} $, and denote $ \pmb{\zeta} :=
(\zeta_{11},\zeta_{12},\zeta_{21},\zeta_{22})\in\R^4$, $ \pmb{\alpha}:=
[\alpha,\beta,\gamma,\delta] \in \RP^3 $. Define the 4 dimensional manifold,
\begin{equation}
  \mathcal{B} :=  \left\{ \left( \pmb{\zeta}, \pmb{\alpha} \right)\in\R^4\times\RP^3 |\ \exists r\in\R\ s.t.\ \pmb{\zeta} = r \pmb{\alpha}  \right\},
\end{equation} 
the \textit{blow-up space} or \textit{blow-up of $ 0 $}.
		
There is a natural projection
\begin{equation}
  \Psi : \mathcal{B} \to \R^4,\quad \left(\pmb{\zeta},\pmb{\alpha}  \right) \to \pmb{\zeta}.
\end{equation}
As $ \Psi $ is differentiable the push forward $ \Psi_*:T\mathcal{B} \to T\R^4 $
induces a vector field $ X_{\mathcal{B}} $ on the blow-up space $ \mathcal{B} $
if we require $ \Psi_*(X_\mathcal{B}) = X $. Note that $ \Psi $ is an
isomorphism outside of $ 0\in\R^4 $ and loses injectivity on the manifold $
\mathcal{C}:= \Psi^{-1}(0) \cong \RP^3 $, often referred to as the
\textit{exceptional divisor}. The primary achievement of the blow-up is to
replace the point at $ 0\in\R^4 $ with the manifold $ \RP^3 $ whilst maintaining
diffeomorphic conjugacy between the flow of $ X $ on $ \R^4\setminus\{0\} $ and
of $ X_{\mathcal{B}} $ on $ \mathcal{B}\setminus\mathcal{C} $.
		
The blow-up method just described can easily be extended to construct a blow-up of the manifold
of simultaneous binary collisions $ \sbc \cong \R^2 \times \hat{\C}^2\times \C^2
=: \mathcal{M} $ with coordinates $(h_1, h_2) \in \R^2$, $(\Gamma_1, \Gamma_2) \in \hat{\C}^2$
and $(x,y) \in \C^2$.
It is done by performing the map $ \Psi $ for each point in $\sbc $ through the map
\begin{equation}
  \Psi\times Id :\mathcal{B}\times\mathcal{M} \to \R^4\times\mathcal{M}.
\end{equation} 
As above, the push forward of $ \Psi\times Id $ gives a vector field $
X_{\mathcal{B}} $ from the vector field in generalised Levi-Civita variables $ X
$. Traditionally in celestial mechanics, the exceptional divisor
\[ \mathcal{C} := (\Psi\times Id)^{-1}(\sbc) \] is named the \textit{collision
  manifold}. The blow-up $ \Psi\times Id $ replaces the set of
collisions $ \sbc \cong \mathcal{M} $ with the higher dimensional collision
manifold $ \mathcal{C}\cong\RP^3\times\mathcal{M} $. Yet, the flow of $
X_{\mathcal{B}} $ off the collision manifold is still conjugate to the flow of $
X $ off the collision set $ \sbc $.
		
The flow on $ \mathcal{C} $ is fictitious in that it does not correspond to
physical orbits of the system. Nevertheless, due to the continuity of the vector
field $ X_{\mathcal{B}} $, orbits near collision are shadowed by orbits on $
\mathcal{C} $. By studying the flow on $ \mathcal{C} $ a topological picture of
a neighbourhood of $ \sbc $ can be formed.

The flow on $ \mathcal{C} $ can be given explicitly by considering charts. Take an
atlas $ \mathcal{A} $ consisting of the 4 charts $
\psi_{\alpha},\psi_{\beta},\psi_{\gamma},\psi_{\delta} $ defined on open sets $
U_\alpha,U_\beta,U_\gamma,U_\delta \subset \mathcal{B} $ which are given by
\begin{equation}
  \begin{aligned}
    \psi_{\alpha}((\zeta_{11},\zeta_{12},\zeta_{21},\zeta_{22}),[\alpha,\beta,\gamma,\delta]) &= \left( \zeta_{11}, \beta/\alpha,\gamma/\alpha,\delta/\alpha \right),\quad U_\alpha = \mathcal{B}\cap\{\alpha\neq 0\} \\
    \psi_{\beta}((\zeta_{11},\zeta_{12},\zeta_{21},\zeta_{22}),[\alpha,\beta,\gamma,\delta]) &= \left( \zeta_{12}, \alpha/\beta,\gamma/\beta,\delta/\beta \right),\quad U_\beta = \mathcal{B}\cap\{\beta\neq 0\} \\
    \psi_{\gamma}((\zeta_{11},\zeta_{12},\zeta_{21},\zeta_{22}),[\alpha,\beta,\gamma,\delta]) &= \left( \zeta_{21}, \alpha/\gamma,\beta/\gamma,\delta/\gamma \right),\quad U_\gamma = \mathcal{B}\cap\{\gamma\neq 0\} \\
    \psi_{\delta}((\zeta_{11},\zeta_{12},\zeta_{21},\zeta_{22}),[\alpha,\beta,\gamma,\delta]) &= \left( \zeta_{22}, \alpha/\delta,\beta/\delta,\gamma/\delta \right),\quad U_\delta = \mathcal{B}\cap\{\delta\neq 0\} \\
  \end{aligned}
\end{equation}
Then, the \textit{$ \zeta_{ij} $-directional blow-ups} $ P_{ij} $ are obtained
by asserting that the following diagram commutes.
\begin{equation}
  \begin{tikzcd}
    (r_{\beta},\alpha_{\beta},\gamma_{\beta},\delta_{\beta})\in\R^4	& \R^4\arrow[l,<-,"P_{12}"]	& (r_{\alpha},\beta_{\alpha},\gamma_{\alpha},\delta_{\alpha})\in\R^4	\\
    \R^4\arrow[d,<-,"P_{21}"]	& \mathcal{B}\arrow[ul,"\psi_{\beta}"]\arrow[u,"\Psi"]\arrow[ur,"\psi_{\alpha}"]\arrow[l,"\Psi"]\arrow[r,"\Psi"]\arrow[dl,"\psi_{\gamma}"]\arrow[d,"\Psi"]\arrow[dr,"\psi_{\delta}"]	& \R^4\arrow[u,<-,"P_{11}"]	\\
    (r_{\gamma},\alpha_{\gamma},\beta_{\gamma},\delta_{\gamma})\in\R^4 &
    \R^4\arrow[r,<-,"P_{22}"] &
    (r_{\delta},\alpha_{\delta},\beta_{\delta},\gamma_{\delta})\in\R^4
  \end{tikzcd}
\end{equation}
Explicitly one computes
\begin{equation}
  \begin{aligned}
    P_{11}(r_\alpha,\beta_{\alpha},\gamma_{\alpha},\delta_{\alpha}) &= r_\alpha(1,\beta_{\alpha},\gamma_{\alpha},\delta_{\alpha})\\
    P_{12}(r_{\beta},\alpha_{\beta},\gamma_{\beta},\delta_{\beta}) &= r_\beta (\alpha_{\beta},1,\gamma_{\beta},\delta_{\beta})\\
    P_{21}(r_{\gamma},\alpha_{\gamma},\beta_{\gamma},\delta_{\gamma}) &= r_\gamma(\alpha_{\gamma},\beta_{\gamma},1,\delta_{\gamma})\\
    P_{22}(r_{\delta},\alpha_{\delta},\beta_{\delta},\gamma_{\delta}) &=
    r_\delta(\alpha_{\delta},\beta_{\delta},\gamma_{\delta},1).
  \end{aligned}
\end{equation}
		
The notation used for coordinates in each chart is designed for quick
computation of inverses and transitions between charts. To see this, identify $
(\zeta_{11},\zeta_{12},\zeta_{21},\zeta_{22}) $ with $
(\alpha,\beta,\gamma,\delta) $ and let $
\xi,\eta,\nu\in\{\alpha,\beta,\gamma,\delta\} $. Then the following formulas
hold,
\begin{equation}
  r_\eta = \eta,\quad \xi_\eta=\xi/\eta,\quad r_\eta \xi_\eta = \xi,\quad \xi_\eta = \xi_\nu /\eta_\nu 
\end{equation}
where $ \xi_\eta = 1 $ if $ \xi = \eta $. For example, $ r_\alpha = \alpha $
which in turn can be identified with $ \zeta_{11} $, that is, $ r_\alpha =
\zeta_{11} $. Similarly, $ \ba =\beta / \alpha = \zeta_{12}/\zeta_{11} $.
		
Each chart $ \psi_\eta \in \mathcal{A} $ is naturally extended to a chart on $
\mathcal{B}\times\mathcal{M} $. Making a distinction between the extension and the
original $ \psi_{\eta} $ will not be needed so the extension is denoted by the
same symbol.
 
The explicit vector fields $ X_\alpha,X_\beta,X_\gamma,X_\delta $ calculated
respectively by the pull-back of $ X $ under $ P_{11*}, P_{12*}, P_{21*},
P_{22*} $ can now be given. In the blow-up procedure one must also desingularise
the vector fields \cite{Ilyashenko2008}. This is done by considering the
rescaled vector fields $ 1/r_\eta X_\eta,\ \eta\in\{\alpha,\beta,\gamma,\delta\}
$, which for convenience will be renamed as $ X_\eta $. We give only the
components of $ X_\eta $ coming from the $ (\zeta_1,\zeta_2) $ components of $ X
$ as the $ (h_1,h_2,\Gamma_1,\Gamma_2,x,y) $ components are computed by mere
replacement of $ (\zeta_1,\zeta_2) $ by the corresponding $ P_{ij} $. Moreover,
only the lowest order terms in $r_\eta$ are given as this is all that is
required for the analysis in this section.

At last, we obtain the blow-up systems \eqref{eqn:twista}~--~\eqref{eqn:twistd}.
We have set $ \Gamma_j = \Gamma_{j1} + i \Gamma_{j2} $. The collision manifold $
\mathcal{C} $ in each chart is given by $ r_\eta = 0 $ for each $
\eta\in\{\alpha,\beta,\gamma,\delta \} $ and is invariant under each of the
flows.
		
		\begin{figure}[!ht]
			\centering
			\begin{subfigure}[b]{0.45\linewidth}
				\begin{equation}\label{eqn:twista}
				\begin{aligned}
					r_\alpha^\prime			&= r_\alpha\left(\gamma_\alpha^2+\delta_\alpha^2\right) + O(r_\alpha^2)\\
					\beta_\alpha^\prime		&=-\beta_\alpha\left(\gamma_\alpha^2+\delta_\alpha^2\right) + O(r_\alpha) \\
					\gamma_{\alpha}^\prime	&= \left(\beta_\alpha^2+1\right)-\gamma_\alpha\left(\gamma_\alpha^2+\delta_\alpha^2\right) + O(r_\alpha) \\
					\delta_{\alpha}^\prime	&= -\delta_\alpha\left(\gamma_\alpha^2+\delta_\alpha^2\right) + O(r_\alpha)
				\end{aligned}
				\end{equation}
			\end{subfigure}
			\qquad
			\begin{subfigure}[b]{0.45\linewidth}
				\begin{equation}\label{eqn:twistb}
				\begin{aligned}
					r_\beta^\prime			&= 0 + O(r_\beta^2) \\
					\alpha_\beta^\prime		&= \left(\gamma_\beta^2+\delta_\beta^2\right) + O(r_\beta) \\
					\gamma_{\beta}^\prime	&= \left(\alpha_\beta^2+1\right) + O(r_\beta) \\
					\delta_{\beta}^\prime	&= 0 + O(r_\beta)
				\end{aligned}
				\end{equation}	
			\end{subfigure}
			~\\[5pt]
			\begin{subfigure}[b]{0.45\linewidth}
				\begin{equation}\label{eqn:twistg}
				\begin{aligned}
					r_\gamma^\prime			&= r_\gamma\left(\alpha_\gamma^2+\beta_\gamma^2\right) + O(r_\gamma^2) \\
					\alpha_\gamma^\prime	&= \left(\delta_\gamma^2+1\right)-|u_2|\alpha_\gamma\left(\alpha_\gamma^2+\beta_\gamma^2\right) +O(r_\gamma)\\
					\beta_{\gamma}^\prime	&= -\beta_\gamma\left(\alpha_\gamma^2+\beta_\gamma^2\right) + O(r_\gamma) \\
					\delta_{\gamma}^\prime	&= -\delta_\gamma\left(\alpha_\gamma^2+\beta_\gamma^2\right) + O(r_\gamma)
				\end{aligned}
				\end{equation}
			\end{subfigure}
			\qquad
			\begin{subfigure}[b]{0.45\linewidth}
				\begin{equation}\label{eqn:twistd}
				\begin{aligned}
					r_\delta^\prime			&= 0 + O(r_\delta^2) \\
					\alpha_\delta^\prime	&= \left(\gamma_\delta^2+1\right) + O(r_\delta) \\
					\beta_{\delta}^\prime	&= 0 + O(r_\delta) \\
					\delta_{\delta}^\prime	&= \left(\alpha_\delta^2+\beta_\delta^2\right) + O(r_\delta)
				\end{aligned}
				\end{equation}	
			\end{subfigure}
		\end{figure}
		
		One must be careful in dealing with the different desingularisations $
    X_\eta $. Even though each $ X_\eta $ is analytic, the different rescaling $
    1/r_\eta $ in each chart prevent them from forming a compatible set of
    vector fields on $ \mathcal{C} $. However, we are not concerned with the
    exact time parameterisation of each orbit, merely the orbit itself. In this
    sense, the various desingularisations form compatible integral curves on $
    \mathcal{C} $. More precisely, instead of considering $ X_\eta $ as inducing
    vector fields on $ \mathcal{C} $, we can think of each $ X_\eta $ as
    inducing line fields on $ \mathcal{C} $. A line field is given by taking a
    line $ l_p $ in the tangent space $ T_p\R^4 $ for each point in a given
    chart. We can take natural line fields $ \bar{X}_\eta $ induced by the
    vector fields $ X_\eta $ through $ \bar{X}_\eta (p) = \Span X_\eta (p) $.
    Each line field creates the analog of integral curves by defining orbits of
    the line field as curves whose tangent at each point is in the line field at
    that point. The desingularised vector fields $ X_\eta $ induce a
    compatible line field on~$\mathcal{C}$.
		 
		Another issue with the desingularisation procedure is the fact that the
    rescaling $ r_\eta^{-1} $ is not strictly positive, forcing the time
    reversal of the vector fields when $ r_\eta^{-1} < 0 $. Hence, in order to
    get a picture of the direction of orbits near $ \mathcal{C} $, it is
    essential to keep track of this time reversal.
		
    % We propose the following method of tracking the time reversal. First, if
    % one is to consider the fictitious flow on $ \mathcal{C} $ as a shadow of
    % the true flow for $ r\neq 0 $, then at a bare minimum one must have the
    % direction of the orbits on $ \mathcal{C} $ reflect the direction of the
    % true orbit under consideration. Observe that a true orbit with $ r_\eta >
    % 0 $ has a shadow on $ \mathcal{C} $ which should flow in the direction
    % provided by $ X_\eta^+ $. Analogously, an orbit with $ r_\eta $ has a
    % shadow with direction provided by $ X_\eta^- $. Hence, we have a method
    % for choosing the correct direction of flow on $ \mathcal{C} $; if the true
    % orbit being shadowed has $ r_\eta > 0 $, then take $ X_\eta^+ $, else,
    % take $ X_\eta^- $. An orbit with $ r_\eta > 0 $ can cross to $ r_\eta < 0
    % $ as $ r_\eta = 0 $ is an invariant manifold. So the choice of $ X_\eta $
    % is unique for each orbit whilst in the chart $ \psi_{jk} $ corresponding
    % to $ \eta $. However, the orbit can swap the sign of $ r_\eta $ when
    % considered in another chart. If this happens, then when returning to the $
    % \eta $ chart, the orbit is mapped to $ r_\eta < 0 $ instead of the
    % original $ r_\eta > 0 $. In such a case, we must then take $ X_\eta^- $.
		
		As an example of when one must be careful, consider the simpler example of 
	the blow-up of $ 0 \in
    \R^2 $ through the map $ \Psi_2 : \mathcal{B}_2 \subset\R^2\times\RP \to
    \R^2 $ defined analogously to the $ \RP^3 $ case already discussed. One
    takes charts by considering the $ x $ and $ y $-directional blow-ups $ (x,y)
    = (r_\alpha,r_\alpha \beta_\alpha) $ and $ (x,y) = (r_\beta \alpha_{\beta},
    r_\beta) $ denoted by $ P_x, P_y $ respectively (in this paragraph $x,y$ are coordinates in the plane unrelated to $(x,y)$ in the planar problem. For more details on this example see \cite{duignanC83regularisationSimultaneous2020}). An orbit near collision
    occurring at $ (x,y) = (0,0) $, the image of the orbit in the two
    directional blow-ups, and some intermediate sections $
    \Sigma_0,\Sigma_1^+,\Sigma_2^+,\Sigma_3 $ are plotted in Figure
    \ref{fig:directionalBlowup}.
		
		The example orbit starts with $ r_\alpha > 0 $ and closely follows the
    collision manifold at $ r_\alpha = 0 $ in forward time, heading toward $
    \beta_{\alpha}= \infty $. The shadowing flow on $ r_\alpha = 0 $ is given
    the correct direction by taking $ X_\alpha $. When one swaps charts from $
    X_\alpha $ to $ X_\beta $ in order to follow the orbit through $
    \beta_{\alpha} = \infty $, the orbit passes from $ \bar{\Sigma}_1^+ $ with $
    \alpha_{\beta} < 0 $ to $ \bar{\Sigma}_2^+ $ with $ \alpha_{\beta} > 0 $.
    Note that $ r_\beta > 0 $, so the compatible flow on $ r_\beta=0 $ is given
    by $ X_\beta $. Finally, when mapped back to the $ \psi_{\alpha} $ chart, $
    r_\alpha $ is now negative and the compatible flow is given by $ -X_\alpha
    $, a time reversed system.
		
		If $ (r_\alpha,\beta_{\alpha}) = (0,0) $ is a hyperbolic saddle, then the
    eigenvalues will swap sign between $ X_\alpha $ and $ -X_\alpha $. In the $
    +X_\alpha $ flow, orbits are pulled toward the collision manifold at $
    r_\alpha = 0 $ and in the $ -X_\alpha $ flow pulled away. This is compatible
    with the true orbit in $ \R^2 $ as it approaches $ 0 $ when $ x > 0 $ $
    (r_\alpha > 0) $ and leaves $ 0 $ when $ x < 0 $ $ (r_\alpha <0) $.
		
		\begin{figure}[ht]		
			\centering \begin{tikzpicture}
\begin{axis}[name = plot1,
scale only axis,
width = 1/4*\textwidth,
height =1/4*\textwidth ,
axis x line=middle,
axis y line=middle,
x label style = {anchor=north},
axis equal,
xlabel = {$x$},
ylabel = {$y$},
restrict y to domain = -4:2,
restrict x to domain = -2:2]
\addplot [domain = 0:0.5, samples = 300]
({sqrt(0.125/x -x^2)}, {x});
\addplot [domain = 0:0.5, samples = 300]
({-sqrt(0.125/x -x^2)}, {x}); 
\addplot [domain = 0:0.5, samples = 300]
(1.5+x^2-x^3, {x}) node[above,pos=1]{$ \Sigma_3 $};
\addplot [domain = 0:0.5, samples = 300]
(-1.5-x^2+x^3, {x}) node[above,pos=1]{$ \Sigma_0 $};
\addplot [domain = 0:1, samples = 300]
({x}, {x}) node[above,pos=1]{$ \Sigma_2 $};
\addplot [domain = 0:-1, samples = 300]
({x}, {-x})node[above,pos=1]{$ \Sigma_1 $};
\end{axis}

\begin{axis}[name = plot2,at={($(plot1.east)+(1.5cm,2cm)$)},anchor=outer west,axis x line=middle,
scale only axis,
width = 1/4*\textwidth,
height =1/4*\textwidth ,
axis y line=middle,
x label style = {anchor = north},
axis equal,
xlabel = {$\alpha_{\beta}$},
ylabel = {$r_\beta$},
restrict y to domain = -2:2,
restrict x to domain = -2:2]
\addplot [domain = 0:0.5, samples = 300]
({sqrt(0.125/x^3 -1)}, {x}) node[above,pos=0.9]{};
\addplot [domain = 0:0.5, samples = 300]
({-sqrt(0.125/x^3 -1)}, {x});
\addplot [domain = 0:0.5, samples = 300]
(1, {x}) node[above,pos=1]{$ \overline{\Sigma}_2 $};
\addplot [domain = 0:0.5, samples = 300]
(-1, {x}) node[above,pos=1]{$ \overline{\Sigma}_1 $};  
\end{axis}

\begin{axis}[name = plot3,at={($(plot1.east)+(1.5cm,-2cm)$)},anchor=outer west,axis x line=middle,
scale only axis,
width = 1/4*\textwidth,
height =1/4*\textwidth ,
axis y line=middle,
x label style = {anchor=north},
y label style = {anchor=east},
axis equal,
xlabel = {$ r_\alpha $},
ylabel = {$ \beta_\alpha $},
ymin = -2,ymax=2,
xmin = -2, xmax = 2]
\addplot [domain = 0.09:1.5, samples = 300]
({x}, {0.25/x}) node[below,pos=0.65]{}; 
\addplot [domain = -0.09:-1.5, samples = 300]
({x}, {0.25/x}) node[above,pos=0.65]{}; 
\addplot [domain = 0:0.5, samples = 300]
(1+x^2-x^3, {x}) node[above,pos=1]{$ \widehat{\Sigma}_3 $};
\addplot [domain = -0.5:0, samples = 300]
(-1-x^2-x^3, {x}) node[below,pos=0]{$ \widehat{\Sigma}_0 $};
\addplot [domain = 0:0.5, samples = 300]
(x, 1) node[above,pos=1]{$ \widehat{\Sigma}_2 $};
\addplot [domain = -0.5:0, samples = 300]
(x, -1) node[below,pos=0]{$ \widehat{\Sigma}_1 $};
\end{axis}

\draw[->] ($(plot1.east) + (0.5cm,0.5cm)$) -- ($ (plot2.south west) +(1cm,1cm)  $) node[midway,above]{$ P_y $};
\draw[->] ($(plot1.east) + (0.5cm,-0.5cm)$) -- ($ (plot3.north west) + (1cm,-1cm)$) node[midway,above]{$ P_x $};

\end{tikzpicture}
			\caption{A plot of a true orbit in $ \R^2 $ and its image in the $ x $ and
        $ y $-directional charts.}
			\label{fig:directionalBlowup}
		\end{figure}
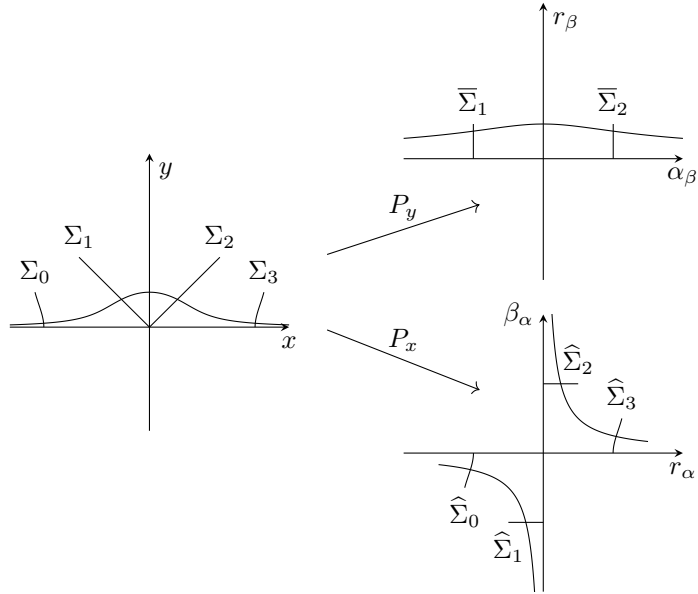
	
		A different representation of $ \mathcal{B}_2 $ is shown in Figure
    \ref{fig:mobiusbundle}. The blow-up space $ \mathcal{B}_2 $ can be thought
    of as the tautological line bundle over $ \RP $, which in turn is
    diffeomorphic to a M\"{o}bius band. This is represented by embedding the
    band in $ \R^3 $. From the resulting figure it can be seen that $
    \mathcal{B}_2 $, considered as a line bundle over $ \RP $, is in fact a
    M\"{o}bius bundle. Consequently, as an orbit follows around the collision
    manifold $ \mathcal{C} \cong \RP \cong S^1 $, the line bundle changes
    orientation. In the $ P_x $ chart, this implies $ r_\alpha \to -r_\alpha $
    as an orbit passes around $ \mathcal{C} $.
		
		\begin{figure}[ht]
			\centering
	\def\svgwidth{0.6\linewidth}
%% Accompanies image file 'mobiusbundle.pdf' (pdf, eps, ps)
\begingroup%
  \makeatletter%
  \providecommand\color[2][]{%
    \errmessage{(Inkscape) Color is used for the text in Inkscape, but the package 'color.sty' is not loaded}%
    \renewcommand\color[2][]{}%
  }%
  \providecommand\transparent[1]{%
    \errmessage{(Inkscape) Transparency is used (non-zero) for the text in Inkscape, but the package 'transparent.sty' is not loaded}%
    \renewcommand\transparent[1]{}%
  }%
  \providecommand\rotatebox[2]{#2}%
  \newcommand*\fsize{\dimexpr\f@size pt\relax}%
  \newcommand*\lineheight[1]{\fontsize{\fsize}{#1\fsize}\selectfont}%
  \ifx\svgwidth\undefined%
    \setlength{\unitlength}{635.26660348bp}%
    \ifx\svgscale\undefined%
      \relax%
    \else%
      \setlength{\unitlength}{\unitlength * \real{\svgscale}}%
    \fi%
  \else%
    \setlength{\unitlength}{\svgwidth}%
  \fi%
  \global\let\svgwidth\undefined%
  \global\let\svgscale\undefined%
  \makeatother%
  \begin{picture}(1,0.78640372)%
    \lineheight{1}%
    \setlength\tabcolsep{0pt}%
    \put(0,0){\includegraphics[width=\unitlength,page=1]{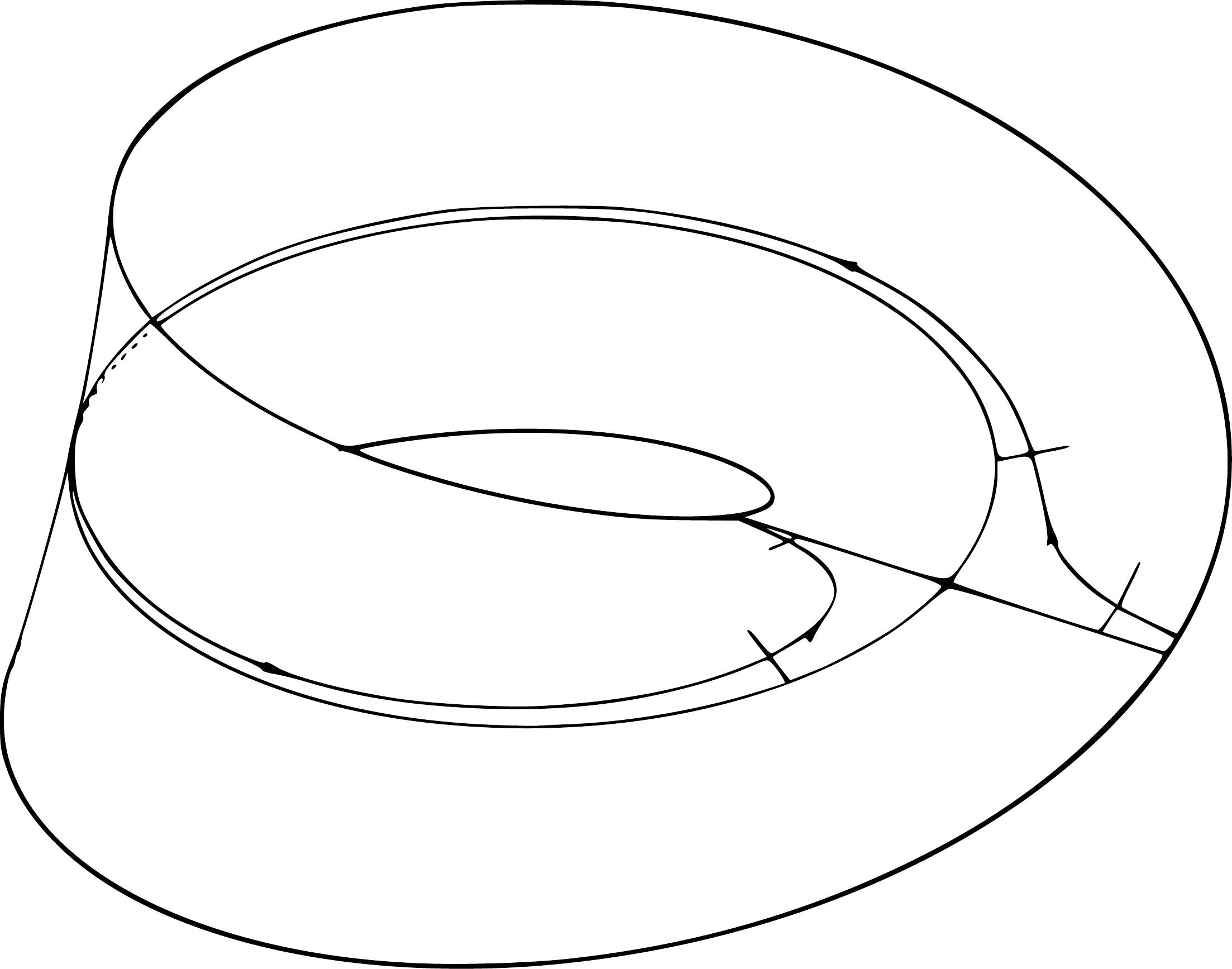}}%
    \put(0.76082016,0.26806717){\color[rgb]{0,0,0}\makebox(0,0)[lt]{\lineheight{2}\smash{\begin{tabular}[t]{l}$ \NHIM $\end{tabular}}}}%
    \put(0.86999405,0.42402894){\color[rgb]{0,0,0}\makebox(0,0)[lt]{\lineheight{2}\smash{\begin{tabular}[t]{l}$ \Sigma_1^+ $\end{tabular}}}}%
    \put(0.92478473,0.33865145){\color[rgb]{0,0,0}\makebox(0,0)[lt]{\lineheight{2}\smash{\begin{tabular}[t]{l}$ \Sigma_0 $\end{tabular}}}}%
    \put(0.55308758,0.27427989){\color[rgb]{0,0,0}\makebox(0,0)[lt]{\lineheight{2}\smash{\begin{tabular}[t]{l}$ \Sigma_2^+ $\end{tabular}}}}%
    \put(0.59141721,0.31620501){\color[rgb]{0,0,0}\makebox(0,0)[lt]{\lineheight{2}\smash{\begin{tabular}[t]{l}$\Sigma_3$\\\end{tabular}}}}%
    \put(0.69127931,0.4737931){\color[rgb]{0,0,0}\makebox(0,0)[lt]{\lineheight{2}\smash{\begin{tabular}[t]{l}$ \mathcal{C} $\end{tabular}}}}%
    \put(0,0){\includegraphics[width=\unitlength,page=2]{mobiusbundle.pdf}}%
  \end{picture}%
\endgroup%

 			\caption{A blow-up of $ \R^2 $ to $ \mathcal{B}_2 $. The collision
        manifold $ \mathcal{C} $ is given by the homoclinic connecting $
        \mathcal{N} $ to itself. A near collision orbit with positive $ x $ is
        plotted in $ \R^2 $ as well as the corresponding pre-image in $
        \mathcal{B}_2 $.}
 			\label{fig:mobiusbundle}
		\end{figure}
		
		In the example Figure \ref{fig:mobiusbundle}, $ \mathcal{N} $ is a
    hyperbolic saddle. However, due to the required time reversing, the
    directions of orbits are not continuous in a neighbourhood of $ \mathcal{N}
    $. The following definition is useful for discussing dynamical objects when
    the exact directions of time are irrelevant.
		\begin{definition}
			Call a singular point $ p $ an \textit{orbital hyperbolic saddle} if, in a
      neighbourhood of $ p $, the flow is orbitally equivalent to the flow of a
      hyperbolic saddle. Similarly define \textit{orbital heteroclinic
        connection}, \textit{orbital focus}, etc.
		\end{definition}
		As an example of this definition, in Figure \ref{fig:mobiusbundle}, one
    would say $ \mathcal{N} $ is an orbital hyperbolic saddle and $ \mathcal{C}
    $ constitutes an orbital homoclinic connection of $ \mathcal{N} $.
    
    After these preliminary consideration in a low-dimensional example where 
    $\mathcal{N}$ is a point instead of a manifold of fixed points and 
    the blow-up is of $\R^2$ instead of $\C^2$ we are now going to describe the 
    dynamics near the simultaneous binary collision in the 4-body problem.
		
    \subsection{Dynamics on the Collision Manifold}
		In this section, the consequences of the blow-up and desingularisation
    procedure are harvested. The flow on the collision manifold $ \mathcal{C} $
    is given by setting $ r_\eta = 0 $ in each $ X_\eta $. Note that in each
    chart the $ (h_1,h_2,\Gamma_1,\Gamma_2,x,y) $ components of the vector field
    can be factored by $ r_\eta $. Each of these components consequently vanish
    on the collision manifold, that is, $(h_1,h_2,\Gamma_1,\Gamma_2,x,y) $ are
    integrals on $ \mathcal{C} $. Geometrically it can be said that the
    collision manifold $ \mathcal{C} $ is foliated by invariant $ \RP^3 $. The
    flow on each $ \RP^3 $ depends on $ \Gamma_1^*,\Gamma_2^* $, but is invariant
    under a choice of values of $ (h_1^*, h_2^*, x^*, y^*) $. A first study of
    the flow on $ \mathcal{C} $ produces the following proposition on its
    topological structure.
		
		\begin{proposition}\label{prop:structureOfCollisionManifold}
			The collision manifold $ \mathcal{C} $ is an orbital homoclinic connection
      of a normally hyperbolic manifold of fixed points, $ \mathcal{N} $, that
      is diffeomorphic to $ \mathcal{M} = \R^2 \times {\hat{\C}}^2 \times \C^2
      $. The following properties of $ \mathcal{N} $ hold:
			\begin{enumerate}[(i)]
      \item $ \mathcal{N} $ is given by the graph $
        (\pmb{\zeta},[\alpha,\beta,\gamma,\delta] ) = (0,
        [1,0,1,0]) $.
				\item The normal bundle of $ \mathcal{N} $ is 4-dimensional in the $
          (r,[\zeta_{11},\zeta_{12},\zeta_{21},\zeta_{22}]) $ directions.
				\item The orbital homoclinic connection is foliated by invariant $ \RP^3
          $.
				\item In the normal bundle $ \mathcal{N} $ is an orbital, resonant
          hyperbolic saddle with (un)stable manifolds depending on the direction
          of time. In one choice, each point in $ \mathcal{N} $ has a
          3-dimensional unstable manifold and a 1-dimensional stable manifold.
          The dimensions are swapped in the alternative choice. Denote this
          1-dimensional manifold at a point $ p\in\mathcal{N} $ by $
          \mathcal{E}_p $.
				\item The 3-dimensional manifold of a point in $ \mathcal{N} $ coincides
          with the homoclinic connection whilst the 1-dimensional manifold $
          \mathcal{E}_p $ is normal to $ \mathcal{C} $.
				\item In the 3-dimensional unstable case, the eigenvalues are of the
          form $ \lambda_1 < 0 < \lambda_2 = \lambda_3 < \lambda_4 $ with ratios
          of hyperbolicity $ -\lambda_1 : \lambda_j $ a constant $ 1:3 $, $1:1$,
          $ 1:1 $ for any fixed point on $ \mathcal{N} $.
        \end{enumerate}
      \end{proposition}
      \begin{proof}
        The desingularised vector field in the chart $ \psi_{\alpha} $ is given
        by setting $ r_\alpha = 0 $ in the vector field $ X_\alpha $. Note when
        $ r_\alpha = 0 $ each of $ r_\alpha^\prime = h_1^\prime = h_2^\prime =
        \Gamma_1^\prime=\Gamma_2^\prime=x^\prime=y^\prime = 0 $. This
        immediately gives the result that $ \mathcal{C} $ is foliated by
        invariant $ \RP^3 $.
			
        Now, for each choice of constants $ (h_1^*, h_2^*,\Gamma_1^*,\Gamma_2^*,
        x^*, y^*)\in\R^2 \times \hat{\C}^2 \times \C^2 ,$ we have precisely one
        equilibrium of $ X_\alpha $ at $
        (\beta_{\alpha},\gamma_{\alpha},\delta_{\alpha}) = (0,1,0) $. This
        equilibrium point maps to the point
        \[ (\pmb{\zeta},[\alpha,\beta,\gamma,\delta],
          h_1,h_2,\Gamma_1,\Gamma_2,x,y) = (0,[1,0,1,0], h_1^*,
          h_2^*,\Gamma_1^*,\Gamma_2^*, x^*, y^* ) =: p^* \] on the collision
        manifold $ \mathcal{C} $. Hence, we have a manifold of fixed points $
        \mathcal{N} $ given by the graph $
        (\pmb{\zeta},[\alpha,\beta,\gamma,\delta]) = (0,
        [1,0,1,0]) $. The Jacobian at
        each point $ p^* $ in the $ \psi_{\alpha} $ chart is given by,
			\begin{equation}
				DX_\alpha|_{p^*}  = 
				\left(\begin{array}{@{}cc@{}}
					\begin{matrix}
						1 	& 0 	& 0	 	& 0 \\
						0	& -1	& 0		& 0 \\
						0	& 0 & -3 & 0 \\
						0	& 0 & 0 & -1 
					\end{matrix}
					& \bigzero \\
					\bigzero & \bigzero
				\end{array}\right).
			\end{equation}
			
			Consequently, at each point $ p^* \in \mathcal{N} $ we have a center
      manifold in the $ h_1,h_2,\Gamma_1,\Gamma_2,x,y $ direction, that is,
      coinciding with the tangent space of $ \mathcal{N} $. In the normal bundle
      of $ \mathcal{N} $, there are 4 non-zero eigenvalues, $ \{1, -1, -3, -1 \}
      $, giving the normal hyperbolicity of $ \mathcal{N} $ and the saddle
      structure. Eigenvectors are given by the basis vectors. It can be
      concluded, for this choice of direction, each point $ p^*\in\mathcal{N} $
      has a 1-dimensional unstable manifold normal to $ \mathcal{C} $ (in the
      direction of $ r_\alpha $ for the $ \psi_{\alpha} $ chart), and a
      3-dimensional stable manifold tangent to the invariant $ \RP^3 $
      corresponding to $ p^* $ (in the $
      (\beta_{\alpha},\gamma_{\alpha},\delta_{\alpha}) $ directions).
			
			The ratios of the eigenvalues are precisely $ 1:1 $, $ 1:3 $, $ 1:1 $ as asserted.
			
			Lastly, because each invariant $ \RP^3 $ is compact with a single
      equilibrium whose stable manifold is 3-dimensional and tangent to $ \RP^3
      $, the stable manifold, considered only as a line field, must in fact be
      all of $ \RP^3 $. That is, every orbit in the collision manifold is a
      homoclinic connection of a point $ p \in \mathcal{N} $.
		\end{proof}
		
		The fact the 1 dimensional manifold $ \mathcal{E}_p $ at each point $ p\in
    \mathcal{N} $ is normal to the collision manifold $ \mathcal{C} $ provides
    some crucial information. Let $\mathcal{U}$ be a tubular neigbourhood of
    $\NHIM$. Due to the time reversal, $ \mathcal{E}_p $ on one side of $
    \mathcal{C} $ in $\mathcal{U}$ is asymptotically approaching $ p $ in
    forward time, whilst the other side is approaching $ p $ in negative time.
    Denote these two halves by $ \mathcal{E}^+_p$ and $ \mathcal{E}^-_p $
    respectively. As the orbit $ \mathcal{E}^+_p $ is approaching the collision
    manifold in forward time it must in fact be a unique collision orbit with
    asymptotic values $ p = (h_1^*,h_2^*,\Gamma_{1}^*,\Gamma_2^*,x^*,y^*) $.
    Likewise, $ \mathcal{E}^-_p $ is an ejection orbit with the same asymptotic
    values $ p $. Hence, we call the entire $ \mathcal{E}_p $ a
    \textit{collision-ejection orbit}. 
    
    By the standard stable manifold theorem
    for normally hyperbolic invariant manifolds, not only is each $
    \mathcal{E}_p $ a smooth manifold, but the bundle $ \mathcal{E} := \cup_p
    \mathcal{E}_p $, called the \textit{collision-ejection manifold}, is smooth
    \cite{fenichelPersistenceSmoothnessInvariant1972,hirschInvariantManifolds1977}.
    When projecting $ \mathcal{E} \subset \mathcal{B}\times\mathcal{M} $ to $
    \R^4\times\mathcal{M} $ via $ \Psi\times Id $, the conjugacy outside of $
    \mathcal{C} $ yields a true manifold with the same properties. The result is
    summarised in the following corollary. It first appears in
    \cite{elbialy1996flow} albeit with a different proof.
		
		\begin{corollary}\label{cor:collisionEjectionIsSmooth}
			Each collision orbit is connected to a unique ejection orbit with the same
      asymptotic values. Moreover, the union of the collision and ejection
      orbits forms a smooth invariant manifold $ \mathcal{E} $.
		\end{corollary}
		The corollary is somewhat visualised in Figures \ref{fig:directionalBlowup}
    and \ref{fig:mobiusbundle}. In the figures, the union of the ingoing and
    outgoing asymptotic orbit is given by a 1-dimensional manifold emanating
    from $ \mathcal{N} $ and normal to $ \mathcal{C} $. The manifold is smooth
    and so to is it's projection into $ \R^2 $.
		
		Corollary \ref{cor:collisionEjectionIsSmooth} leads to another neat
    consequence of Proposition \ref{prop:structureOfCollisionManifold}. Once the
    existence of a $ C^0 $ block map $ \pib $ is established in Theorem
    \ref{thm:C0Regularisable} below, the smoothness of $ \mathcal{E} $
    guarantees that this $ \pib $, restricted to $ \mathcal{E} $, is smooth. The
    following corollary is immediate.
		\begin{corollary}\label{cor:subproblemsRegularisable}
			Consider a sub-problem $ \mathcal{P} $ of the planar $ 4 $-body problem
      which is entirely contained within the collision-ejection manifold $
      \mathcal{E} $. Inside of $ \mathcal{P} $, the set of simultaneous binary
      collisions are $ C^\infty $-regularisable.
		\end{corollary}
		In particular, the rhomboidal and symmetric collinear sub-problems are $
    C^\infty $-regularisable. For details on these sub-problems see, for
    instance, \cite{alvarez-ramirezReviewPlanarCaledonian2014}. This observation
    agrees with the regularisation results of
    \cite{Bakker2011,sekiguchiSymmetricCollinearFourBody2004,sivasankaranGlobalRegularisationIntegrating2010}.
	
	\subsection{Integrability of the flow on the Collision Manifold}\label{sec:untwistingIntegrals}

		In the $ \zeta_1,\zeta_2 $ coordinates some key properties of the flow on
    the collision manifold become clear. Note that $ \beta_{\delta},
    \delta_{\beta} $ are local integrals of the flow on the collision manifold
    in the respective charts $ \psi_{\delta}, \psi_{\beta} $. As remarked in
    Proposition \ref{prop:structureOfCollisionManifold}, so too are $
    h_1,h_2,\Gamma_1,\Gamma_2,x,y $. In fact, the flow in each desingularised
    vector field $ X_\eta $ is integrable.
		
		To see this, recall that $ L_j := \im(\zeta_j) =
    \frac{1}{2}a_j^{1/3}\im(\bar{u}_j \tilde{z}_j ) $ is proportional to the
    intrinsic angular momentum of the $ j^{th} $ distressed binary,
    $\frac{1}{2}\im(\bar{u}_j\tilde{z}_j)$. Projecting this (scaled) intrinsic
    angular momentum $ L_j $ onto $ \mathcal{C} $ results in the projective
    coordinate $ [0,\beta,0,0] $ for $ j=1 $ and $ [0,0,0,\delta] $ for $ j=2 $.
    The ratio $ L_2/L_1 $ is invariant under the blow-up $ \Psi $ and maps onto
    $ \mathcal{C} $ as $ \kappa_1 := \delta/\beta $. Hence $ \kappa_1 $ is an
    integral.
		
		There is a second integral $ \kappa_2 $ related to the original time of the
    system $t$. Take the leading order terms in $ (\zeta_1,\zeta_2) $ of system
    \eqref{eqn:generalisedLeviCivita}, 
		\[ \dot{\zeta}_j = |\zeta_j|^{-2} \implies 
		\begin{aligned}
			\dot{I}_j &= (I_j^2 + L_j^2)^{-1} \\
			\dot{L}_j &= 0,
		\end{aligned} \] with $\zeta_j = I_j + i L_j$.

  Clearly each $ L_j $ is an integral of the leading terms. This observation
  leads to the remarks above on $ \kappa_1 $. But also the equation for $ I_j $
  can be integrated to give $ 3 t - 3 t_0 = I_j^3 + 3 I_j L_j^2 $ for $ j= 1,2
  $. Taking the equation for $ j=1 $ and subtracting for $ j=2 $ yields $ (I_1^3
  + 3 I_1 L_1^2) - (I_2^3 + 3 I_2 L_2^2) $. As argued above, taking the ratio
  with $ L_1^3 $ or $L_2^3$, this approximate integral descends to a true integral $
  \kappa_2 $ on the collision manifold $ \mathcal{C} $.
		
  After taking care of some technicalities, we will now show the integrability of
  the line field on the collision manifold $ \mathcal{C} $.

		\begin{definition}
			We say a manifold $ \mathcal{K} \subset \mathcal{C} $ is \textit{orbitally
        invariant} if the image in each chart $ \psi_{\eta}(\mathcal{K}),\ \eta
      \in \{\alpha,\beta,\gamma,\delta \} $ is an invariant manifold of the
      desingularised vector fields $ X_\eta $.
		\end{definition}
		\begin{proposition}\label{prop:CollisionManifoldInvariants}
			Define the two smooth functions $ \hat{\kappa}_1,\hat{\kappa}_2 :
      \mathcal{C}\cong \RP^3 \to \RP $ by,
			\begin{align*}
				\hat{\kappa}_1([\alpha,\beta,\gamma,\delta]) &= [\beta,\delta] \\
				\hat{\kappa}_2([\alpha,\beta,\gamma,\delta]) &= [\beta^3, (\alpha^3 + 3 \beta^2 \alpha) - (\gamma^3 + 3 \delta^2 \gamma)].
			\end{align*}
			Then for each $ w,v \in \RP $, the closures of the pre-images $
      \overline{\hat{\kappa}_1^{-1}(w)},\overline{\hat{\kappa}_2^{-1}(v)} $ are
      smooth, orbitally invariant manifolds that intersect transversally.
		\end{proposition}
		\begin{proof}
			Let $ w = [\beta^*, \delta^*] \in \RP $ and consider $
      \hat{\kappa}_1^{-1}(w) = \{ [\alpha,\beta,\gamma,\delta] | [\beta,\delta]
      = [\beta^*,\delta^*] \} $. Clearly $ \hat{\kappa}_1^{-1}(w) $ is a smooth
      submanifold of $ \mathcal{C} $. For each $ \eta \in \{
      \alpha,\beta,\gamma,\delta \} $, $ \hat{\kappa}_1^{-1}(w) $ in the $
      \psi_{\eta} $ chart is
			\[ \{ (\alpha_{\eta}, \beta_{\eta},\gamma_{\eta},\delta_{\eta} ) |
        \delta_\eta/\beta_{\eta} = \kappa_1 \in \R\cup\{ \infty \},\
        (\beta_{\eta},\delta_\eta) \neq (0,0) \} \subset \R^3, \] where $
      \eta_\eta = 1 $ is removed from $ (\alpha_{\eta},
      \beta_{\eta},\gamma_{\eta},\delta_{\eta} ) $. A quick calculation reveals
      these level sets of $ \delta_\eta/\beta_{\eta} $ are invariant in each
      chart. Each $ \delta_\eta/\beta_{\eta} = \kappa_1 $ is a 2-plane in $ \R^3
      $ minus the line $ \beta_\eta=\delta_\eta=0 $. Taking the closure, we
      obtain a complete 2-plane which is clearly smooth.
			
			Similarly, $ \hat{\kappa}_2^{-1}(w) $ in each $ \psi_{\eta} $ chart is given by,
			\[ \{ (\alpha_{\eta}, \beta_{\eta},\gamma_{\eta},\delta_{\eta} ) |
        ((\alpha_\eta^3 + 3 \beta_\eta^2 \alpha_\eta) - (\gamma_\eta^3 + 3
        \delta_\eta^2 \gamma_\eta))/\beta_\eta^3 = \kappa_2 \in \R\cup\{\infty
        \},\ (\beta_\eta,\alpha_\eta^3-\gamma_\eta^3 -3\delta_\eta^2
        \gamma_\eta) \neq 0 \} \subset \R^3. \] where again $ \eta_\eta = 1 $ is
      removed. A quick calculation in each $ X_\eta $ reveals $ \kappa_2 $ is
      invariant. Further, define the function $ F_2(\alpha_{\eta},
      \beta_{\eta},\gamma_{\eta},\delta_{\eta}) = ((\alpha_\eta^3 + 3
      \beta_\eta^2 \alpha_\eta) - (\gamma_\eta^3 + 3 \delta_\eta^2 \gamma_\eta))
      - \kappa_2 \beta_\eta^3 $. Then
			\[ DF_2 = \left(\begin{array}{cccc}
				3(\alpha_\eta^2 + \beta_\eta^2) & 6 \alpha_\eta \beta_\eta - \kappa_2 \beta_\eta^2 & -3(\gamma_\eta^2 + \delta_\eta^2) & - 6 \gamma_\eta \delta_\eta
			\end{array}\right). \]
  Now, at least one of $ \alpha_{\eta}, \beta_{\eta},\gamma_{\eta},\delta_{\eta}
  $ must be equal to 1. Therefore, $ DF_2 \neq 0 $ in any of the charts. By the
  implicit function theorem, the closure of each $ \hat{\kappa}_2^{-1}(w) $ is a
  smooth submanifold of $ \mathcal{C} $.
			
  Lastly, defining $ F_1(\alpha_{\eta},
  \beta_{\eta},\gamma_{\eta},\delta_{\eta}) = \delta_\eta - \kappa_1 \beta_\eta
  $, computing \[ DF_1 = \left(\begin{array}{cccc} 0 & -\kappa_1 & 0 & 1
			\end{array}\right), \]
			and comparing $ DF_2 $ to $ DF_1 $ along $ F_1 = 0, F_2 = 0 $, it is seen that $ DF_1 \neq DF_2 $ at any mutual point. That is, $ \overline{\hat{\kappa}_1^{-1}(w)},\overline{\hat{\kappa}_2^{-1}(v)} $ intersect transversally.
		\end{proof}
		
		Define the function $ \iota:\RP \to \R\cup\{\infty \} $ by $ [\alpha,\beta]
    \mapsto \beta/\alpha $ and the two functions $ \kappa_1 :=
    \iota\circ\hat{\kappa}_1, \kappa_2 := \iota \circ \hat{\kappa}_2 $. Due to
    Proposition \ref{prop:CollisionManifoldInvariants}, each of the level sets
    of $ \kappa_1,\kappa_2 $ must define smooth invariant manifolds. With the
    intersection between any two level sets of $ \kappa_1,\kappa_2 $ transverse,
    we obtain the integrability of the flows as a corollary.
		\begin{corollary}
			Each desingularised vector field $ X_\delta $ is integrable with first
      integrals given by the images of $
      \kappa_1,\kappa_2,h_1,h_2,\Gamma_1,\Gamma_2,x,y $ in each chart.
		\end{corollary}
	
	\subsection{Proof of $ C^0 $-regularity}

		Finally, we are in a position to give a new proof of the $ C^0 $-regularity of $ \sbc $.
		\begin{thm}\label{thm:C0Regularisable}
			The set of simultaneous binary collisions $ \sbc $ is at least $ C^0 $-regularisable in the planar 4-body problem.
		\end{thm}
		\begin{proof}
			In order to show $ C^0 $-regularity of $ \sbc $ we must first define two
      sections; $ \Sigma_0 $ transverse to the collision orbits $ \mathcal{E}^+
      $, and $ \Sigma_3 $ transverse to ejection orbits $ \mathcal{E}^- $. Then,
      by flowing points from $ \Sigma_0 \setminus \mathcal{E}^+ $ to $ \Sigma_3
      \setminus \mathcal{E}^- $ we obtain a homeomorphism $ \pi $. The $ C^0
      $-regularity of $ \sbc $ is guaranteed provided one can extend $ \pi $ to
      some $ \pib:\Sigma_0 \to \Sigma_3 $ such that $ \pib $ is unique and $ C^0
      $. The blown-up systems $ X_\eta $ will be essential to prove
      this statement.
				
			In the blown-up systems $ X_\eta $, take a section $ \Sigma_0 $
      transverse to $ \mathcal{E}_p^+ $ for some point $ p \in \mathcal{N} $.
      Recall from the remarks in Section \ref{sec:CollisionManifold}, the
      desingularisation $ X_\eta $ or $- X_\eta $ that is consistent with the
      flow outside of $ \mathcal{C} $ must be taken. In order to see which time
      direction to take, recall that $ \mathcal{N} $ is
      given by the graph $(\pmb{\zeta}, [\alpha,\beta,\gamma,\delta] =
      (0,[1,0,1,0]) $. Hence, $ \mathcal{E}_p^+ $, in some tubular neighbourhood
      of $ \mathcal{C} $, is entirely contained in the chart $ \psi_{\alpha} $.
      From Proposition \ref{prop:structureOfCollisionManifold}, $
      \mathcal{E}_p^+ $ is the unstable manifold of a point $ p\in\mathcal{N} $
      which is normal to $ \mathcal{C} $. It can be concluded, in order to have
      the orbit $ \mathcal{E}^+_p $ approach the collision in forward time as
      collision orbits should, we must consider the desingularisation $
      -X_\alpha $. Moreover, as one only needs to consider $ -X_\alpha $ for
      orbits with $ r_\alpha < 0 $, and $ r_\alpha $ is mapped under $ P_{11} $
      to $ \zeta_{11} $, we must have that collision orbits approach the set of
      simultaneous binary collisions $ \sbc $ with $ \zeta_{11} < 0 $. A
      symmetrical argument shows for ejection orbits $ r_\alpha >0 \implies
      \zeta_{11} > 0 $ and the correct desingularisation is $ X_\alpha $.
			
			Now, orbits on $ \Sigma_0 $ follow the stable manifold $ \mathcal{E}_p^+ $
      until they pass through a normally hyperbolic region close to $
      p\in\mathcal{N} $. From Proposition
      \ref{prop:structureOfCollisionManifold}, this point is a normally
      hyperbolic saddle and, in the $ -X_\alpha $ desingularisation, has a
      1-dimensional stable manifold $ \mathcal{E}_p^+ $ and a 3-dimensional
      unstable manifold coinciding with the 3-dimensional invariant $ \RP^3 $ at
      the fixed value $ h_1^*,h_2^*,\Gamma_1^*,\Gamma_2^*,x^*,y^* $. It is known
      that normally hyperbolic manifolds are topologically equivalent to their
      linear part \cite{palis1977topological}. Hence, we can conclude that near
      collision orbits around $ \mathcal{E}^+_p $ will get pulled away from $
      \mathcal{E}^+_p $ and begin to follow orbits on the unstable manifold of $
      p $, that is, orbits on the collision manifold $ \mathcal{C} $. That is,
      after the hyperbolic region near $ \mathcal{N} $ the near collision orbits
      will be shadowed by orbits on $ \mathcal{C} $.
			
			By Proposition \ref{prop:structureOfCollisionManifold}, each orbit on the
      collision manifold is a homoclinic connection of the point $
      p\in\mathcal{N} $. As the orbit on $ \mathcal{C} $ passes around the
      manifold back to $ p $, the fact that $ \mathcal{C} \cong \RP^3
      \times\mathcal{M} $ and the non-orientability of $ \RP^3 $ ensure that,
      for the near collision orbits, $ r_\alpha $ changed orientation from $
      r_\alpha < 0 $ to $ r_\alpha > 0 $. That is, if an orbit passes near
      collision it must pass through $ \Sigma_0 $ where $ \zeta_{11}< 0 $ and
      exit on $ \Sigma_3 $ where $ \zeta_{11} > 0 $. When $ r_\alpha > 0 $, we
      need to consider the desingularisation $ X_\alpha $ where $
      \mathcal{E}_p^- $ is the unstable manifold of $ p $. The near collision
      orbits then pass back through a hyperbolic region where they begin to
      follow the orbit $ \mathcal{E}_p^- $. Each of these orbits then intersects
      some $ \Sigma_3 $, a transverse section to the ejection orbit $
      \mathcal{E}_p^- $.
			
			By limiting to the orbit $ \mathcal{E}_p^+ $ on $ \Sigma_0 $, we obtain
      near collision orbits which pass around the collision manifold and limit
      onto the unique ejection orbit $ \mathcal{E}_p^- $. Using Corollary
      \ref{cor:collisionEjectionIsSmooth} and the topological conjugacy of the
      blow-up $ \mathcal{B}\times\mathcal{M} $ outside of $ \mathcal{C} $, we
      conclude the $ C^0 $ regularity of the simultaneous binary collisions
      for the planar problem.
		\end{proof}
	
	\section{$ C^{8/3} $-regularity of the Block Map}\label{sec:Ckregularisation}
This section contains the proof of the main result of the paper, 
Theorem~\ref{thm:C2Regularisation} on the $ C^{8/3} $-regularisation of 
the simultaneous binary collision of the planar 4-body problem.
	 
\subsection{Normal Form and Absence of Foliation}

In previous work on the collinear problem \cite{duignanC83regularisationSimultaneous2020}, a link was
established between the $C^{8/3}$-regularity of the simultaneous binary
collision and the inability to foliate a tubular neighbourhood of the collinear
collision manifold by invariant manifolds normal to the manifold. A method which
quantifies how well the collinear collision manifold admits an invariant
foliation was developed. The normal form of $ X $ in a neighbourhood of $\sbc$
was crucial to the quantification. Specifically, this inability to foliate the
neighbourhood was shown to be a consequence of resonant terms at order $ 8 $ in
the normal form of $ X $ in a neighbourhood of $\sbc$. In this section, we use
similar arguments to show the normal space to $ \sbc $ in the planar problem
also lacks a smooth foliation. The result and the proceeding calculations in
Section \ref{sec:GeometricSketch} will prove the $C^{8/3}$ regularity.
		
\subsubsection{Normal Form Theory}
Recall the following definitions and theorem on normal forms which combines the
works of Belitskii \cite{belitskii1979invariant,belitskii2002c}, and Stolovitch
and Lombardi \cite{lombardi2010normal}.
		
\begin{definition}
  Decompose a vector field $ X $ into its Taylor series,
  \[ X = X_0 + X_1 + \dots \] with $ X_0 $ the leading order homogeneous
  component of degree $ s $ and each $ X_{d} \in \mathcal{H}_{d+s-1} $ in the space
  of degree $ d + s $ homogeneous vector fields.
  Let $ [\cdot,\cdot] $ be the usual Lie bracket for vector fields and define
  the \textit{cohomological operator}
  \[ L_{X_0} := [X_0,\cdot] \] and its restriction to $ \mathcal{H}_{d} $ by 
  $L_d : \mathcal{H}_d \mapsto \mathcal{H}_{s + d - 1}$.
  % Moreover, define the adjoint of each $ L_d $ with respect to the
  % Fischer inner product by $ L_d^* $. Details on the Fischer inner product are
  % given in \cite{belitskii2002c}.
\end{definition}

Most treatments of normal form theory are focused on the case $s = 1$,
i.e.~equilibrium points that have a non-vanishing linear part.
Note that for the simultaneous binary collision in the 4-body problem 
after the time re-scaling equilibrium points with vanishing linear part 
are obtained. The general case with $s=2$ needed here is, e.g., described in \cite{lombardi2010normal}.
The following normal 
form theorem is due to \cite{belitskii1979invariant,lombardi2010normal}
for a particular choice of $\mathcal{U}_d$ described below. We need a straightforward generalisation 
which allows for a more general choice of $\mathcal{U}_d$.

\begin{thm}\label{thm:FormalNF}
  For each $d \geq 1$, let $\mathcal{U}_d$ be a subspace of $\mathcal{H}_d$
  such that  $\mathcal{H}_d =  \Ima{L_d} + \mathcal{U}_d$. 
  Then there exists a formal transformation $
  \hat{\phi}^{-1} = I + \sum_{d\geq 1} U_d $ with $ U_d\in \mathcal{H}_{d} $
  that formally conjugates $ X = X_0+\sum_{d\geq 1} X_d $ to the normal form,
  \begin{equation}
    \hat{\phi}^*X = X_0 + \sum\limits_{d\geq 1} N_d,
  \end{equation}
  such that $ N_d \in \mathcal{U}_d$. 
\end{thm}

The version of Theorem~\ref{thm:FormalNF} as proved in \cite{belitskii1979invariant,lombardi2010normal} 
takes $\mathcal{U}_d$ as a complement of $\Ima{L_d}$ defined 
using the Fischer inner product on $\mathcal{H}_d$. Specifically, using an inner product, the adjoint $L_d^*$
of $L_d$ can be defined, and a complementary subspace can be chosen as $\ker(L_d^*)$.
Details on the Fischer inner product are given in \cite{belitskii2002c}. 
Our work on the collinear problem \cite{duignanC83regularisationSimultaneous2020}
follows this convention. There is no requirement to making the same choice here
for the planar problem. However, if one chooses to do so one will
discover at least two unfavourable outcomes.

The first unfavourable outcome is that the normal form procedure fails to compute higher
order integrals of $L_1,L_2$. The ability to compute these integrals allows for
ease in the computations of the block map in the proceeding section. The second
unfavourable outcome is that the resulting normal form fails to limit to that computed in
\cite{duignanC83regularisationSimultaneous2020} when restricted to the collinear problem.

These concerns can be treated by making a non-standard choice for $\mathcal{U}_d$,
which does not happen to be complementary subspace to $\Ima{L_d}$.
Depending on the choice of $\mathcal{U}_d$ this may remove fewer terms in 
the normal form. Depending on the situation, this may be favourable.

In our current situation, instead of taking $\mathcal{U}_d = \ker(L_d^*)$, 
we will take  $\mathcal{U}_d = \ker(L_{X_c}^*)|_{\mathcal{H}_d}$,
where $X_c$ contains the leading
order terms of $X$ restricted to the collinear problem,
\[X_c = I_2^2\partial_{I_1} + I_1^2 \partial_{I_2} = \left. X_0 \right|_{L_1 = L_2 = 0}.\] 
To define the adjoint we still use the Fischer inner product, so that 
\[X_c^* = I_2 \partial_{I_1}^2 + I_1 \partial_{I_2}^2.\]
The novelty in our treatment is that we do not use the adjoint of the homological operator to 
define a complementary subspace, but the adjoint of a different operator obtained by restriction 
to the invariant subspace $L_1 = L_2 = 0$.
The particular choice of $X_c$ is supported by the following lemma.

% \comment{Relation between $X_0$ and $X_c$ is simple, what does this imply for relation between  $L_{X_c}^*$ and $L_{X_0}^*$?}

\begin{lemma}
  For each $d \leq 15$ the space of degree $d+1$ homogenous vector fields
  $\mathcal{H}_d$ can be written as
  \[\mathcal{H}_d = \Ima(L_d) + \ker(L_{X_c}^*).\]
\end{lemma}
\begin{proof}
  This can be verified by explicit computation.
\end{proof}
Unfortunately we do not have a general proof for all $d \in \N$.

\subsubsection{Calculation of the normal form}
With a suitable choice of subspace $\mathcal{U}_d$, we may now proceed
with the calculation of the normal form. Let $ x = x_1 + i x_2, y= y_1 + i y_2,
$ and $ X_\R $ be the real vector field associated to $ X $, the vector field
near $ \sbc $ in the generalised Levi-Civita coordinates and with fictitious
time $\tau$. In order to construct $ X_\R $ a choice of coordinates on $
\hat{\C} $ for each $ \Gamma_j $ must also be made. The proceeding normal form
calculation is independent of the choice, so we refrain from making any
specifications. The realisation of $ X$ is hence $ X_\R $, a vector field on $
\R^{12} $.
	
The leading order term at any point in $ \sbc $ is given by $ X_0 = |\zeta_2|^2 \partial_{I_1} + |\zeta_1|^2 \partial_{I_2}$.
Let $ w = (w_1,\dots,w_{12}) \in \mathcal{H}_{d+1} $ and denoting $X_c =
I_2^2\partial_{I_1}+I_1^2\partial_{I_2} $ the leading order vector field $X_0$
restricted to the collinear subspace $L_1 = L_2 = 0$. 
The operator $L_{X_{c}}^*$ is given explicitly by,
\begin{equation}\label{eqn:cohomop}
  L^{*}_{X_c} w = {X}_{c}^*(w) - \left(DX_{c}^*\right)^Tw,\quad DX_{c}^* = \left( \begin{array}{cccc}
                                                                                            0				& 0					& 2 \partial_{I_2}	&  0 	\\
                                                                                            0				& 0					& 0					& 0					\\
                                                                                            2\partial_{I_1}	&  0	& 0					& 0					\\
                                                                                            0				& 0					& 0					& 0					\\
                                                                                          \end{array} \right)\oplus 0_8,
\end{equation} 
with $ 0_8 $ the $ 8\times 8 $ zero matrix.

Note that the cohomological operator decouples into an operator in $
I_1,L_1,I_2,L_2 $ and merely ${X}^*_{c}$ in each of the remaining
variables. This is a consequence of the fact that $ X_{0} $ decouples into the $
\zeta_1,\zeta_2 $ system and a trivial vector field in the other variables. The
decoupling  leads to a proof of Corollary \ref{cor:noFoliatation} on
the inability to construct an invariant foliation of the normal space to $ \sbc
$ in the planar problem.
		
The normal form near an arbitrary simultaneous binary collision to degree $ 9 $
will now be computed. For this calculation we require the degree $ 9 $ Taylor
expansion of the vector field $ X $ around an arbitrary fixed point $ p =
(0,0,h_{1}^*,h_{2}^*,\Gamma_1^*,\Gamma_2^*,x^*, y^*) $ in $ \sbc $.
Consequently, the potential $ K(\zeta_1,\zeta_2,h_1,h_2,\Gamma_1,\Gamma_2,x) $
must be expanded to degree $ 8 $.
		
Firstly, there are four terms in the original $ \hat{K}(Q_1,Q_2,x) $ which are
of the form,
\[ \frac{d_i}{|x+ C_l Q_1 + C_l Q_2 |} = \frac{d_i}{|x|} \frac{1}{|1+C_l Q_1/x +
    C_m Q_2/x|}. \] In order to get $ K $ to degree 9 we need to compute this
expansion to degree 4 in $ Q_1,Q_2 $ and substitute the various coordinate
transformations of Section \ref{sec:coordinates}. Consider the function,
\begin{equation}
  F(\xi) = (1+\xi)^{-1/2},\quad F:\C\to\C,
\end{equation}
which is holomorphic away from $ \xi = -1 $, in particular, in a neighbourhood
of 0. As it is holomorphic, it has a convergent Laurent series about $ 0 $ given
by
\begin{equation}\label{eqn:FExpansion}
  F(\xi) = \sum\limits_{j=1}^{\infty} (-1)^j\binom{-1/2}{j} \xi^j.
\end{equation}
The potential can then be written in the form,
\begin{equation}
  \begin{aligned}
    \hat{K}(Q_1,Q_2,x) &= \frac{1}{|x|}\left( d_1\left|F\left(c_2\frac{Q_1}{x}-c_4\frac{Q_2}{x}\right)\right|^2 + d_2\left|F\left(c_2\frac{Q_1}{x}+c_3\frac{Q_2}{x}\right)\right|^2 \right. \\
    &\quad\left. +
      d_3\left|F\left(-c_1\frac{Q_1}{x}-c_4\frac{Q_2}{x}\right)\right|^2 +
      d_4\left|F\left(-c_1\frac{Q_1}{x}+c_3\frac{Q_2}{x}\right)\right|^2
    \right).
  \end{aligned}
\end{equation}
Using $ |F(\xi)|^2 = F(\xi)F(\bar{\xi}) $ and \eqref{eqn:FExpansion}, we can
expand $ \hat{K}(Q_1,Q_2,x) $ in a neighbourhood of $ Q_1,Q_2 = 0, x\neq 0 $.
Substituting the values of $ d_i,c_i $ given in \eqref{eqn:massconsts} this
expansion to degree $ 4 $ of $ \hat{K} $ is given as,
\begin{equation}\label{eqn:potential}
  \begin{aligned}
    \hat{K}(Q_1,Q_2,x) &= \frac{1}{|x|}\left( b_{0} + \hat{K}_1 \left( \frac{Q_1}{x} \right) + \hat{K}_2\left( \frac{Q_2}{x} \right) + \hat{b}_{c} W_{c}\left(\frac{Q_1}{x},\frac{Q_2}{x} \right)  \right),\\
    \hat{K}_i(Q) &= \sum_{j=2}^{4} \hat{b}_{ij} W_j\left(Q \right), \\
  \end{aligned}
\end{equation}
where $ \hat{b}_{ij}, \hat{b}_c, b_0 $ are functions of the masses, $ W_j $ are
homogeneous degree $ j $ polynomials in $ Q, \overline{Q} $ independent of the
masses, and $ W_c $ is a homogeneous degree 4 polynomial independent of the
masses and containing all the degree 4 coupled terms between $ Q_1 $ and $ Q_2
$. The constants $b_{ij}$ and the homogeneous polynomials are given in Appendix
\ref{sec:Constants}.
		
\begin{remark}
  A priori, there should be coupled terms of lower order; for example the
  monomial $ Q_1 Q_2 $ of order 2. Remarkably, for the particular potential $
  \hat{K} $, all these terms vanish. The first coupled monomials are at order 4
  in $ \hat{b}_c \hat{W}_c(Q_1/x,Q_2/x) $. It will be seen that these first
  coupling terms play a crucial role in the arrival of non-vanishing resonant
  terms and ultimately the finite differentiability of the block map. In the
  collinear problem the crucial role of the coupling terms was heuristically
  observed by Martinez and Sim\'{o} \cite{Martinez1999}.
\end{remark}
	
Each of the transformations in Section \ref{sec:coordinates} can be carried
through Equation \eqref{eqn:potential} to get expansions to order 9 of $ K $. It
will take the form,
\begin{equation}\label{eqn:zPotential}
  \begin{aligned}
    K(\zeta_1,\zeta_2,x) &= \frac{1}{|x|}\left( b_{0} + K_1 \left( \frac{U_1^{-2}\Gamma_1^2\zeta_1^2}{x} \right) + K_2\left( \frac{U_2^{-2}\Gamma_2^2 \zeta_2^2}{x} \right) + b_{c} W_{c}\left(\frac{U_1^{-2}\Gamma_1^2\zeta_1^2}{x},\frac{U_2^{-2}\Gamma_2^2\zeta_2^2}{x} \right)  \right),\\
    K_i(Q) &= \sum_{j=2}^{4} b_{ij} W_j(Q/x), \\
  \end{aligned}
\end{equation} 
with $ b_{ij}, b_c $ given in Appendix~\ref{sec:Constants}.

% \comment{I agree that one can scale and rotate, but it may be helpful to see where $x$ actually appears, or at least say something about it in words? Also, if we use the scaling % symmetry, then we should also state it in the beginning?}
The following result on the normal form near an arbitrary simultaneous binary
collision can now be given. 
The potential is expanded under the assumption that $|x|$ is large compared 
the distances of the distressed binaries $|Q_i|$. In the expansion of the potential 
thus $x$ appears in the denominator of all terms.
Due to the scaling and rotational symmetry of the
Hamiltonian $ H $, see~\eqref{eqn:scal}, 
it can be assumed that $ x $ has the asymptotic value $ x^* = 1 $.
From now on we assume that the scaling and rotation that achieves this has been fixed,
and hence the $x$ in the denominator of the expansion disappears.

After the Taylor expansion the terms with the smallest degree in $\zeta_i$
in the components of the vector field \eqref{eqn:generalisedLeviCivita} have degree 
$(2, 8, 2, 8, 5, 5, 3, 3, 4, 4, 4, 4)$.
As a result of the time scaling the leading order terms are always proportional to either 
$|\zeta_1|^2$, $|\zeta_2|^2$, or both, so that $\zeta_1 = \zeta_2 = 0$ makes the vector field
vanish, irrespective of the values of $h_1, h_2, \Gamma_1, \Gamma_2, x, y$.
The normal form transformation produces a vector field $X^9$ where the leading degrees are 
$(2,0,2,0,9,9,0,0,0,0,0,0)$, where 0 represents a vanishing component up to including order 9.

\begin{proposition}\label{prop:normalformforSBC}
  The normal form $ X^9 $ of the vector field
  \eqref{eqn:generalisedLeviCivita} in a neighbourhood of the simultaneous
  binary collision with asymptotic values $(h_{1},h_{2},\Gamma_1,\Gamma_2,x,y) =
  (h_{1}^*,h_{2}^*,\Gamma_1^*,\Gamma_2^*,1, y^*)$ is given to degree 9 by,
  \begin{equation}\label{eqn:preBlowUpNormalForm}
    \begin{aligned}
      I_1^\prime	&= |\zeta_2|^2 + R_1^{4}(\zeta_1,\zeta_2,h_1,h_2) + R_1^{6}(\zeta_1,\zeta_2,h_1,h_2) + R_1^{8}(\zeta_1,\zeta_2,h_1,h_2) \\
      L_1^\prime	&= 0 \\
      I_2^\prime	&= |\zeta_1|^2 + R_1^{4}(\zeta_2,\zeta_1,h_2,h_1)+ R_1^{6}(\zeta_2,\zeta_1,h_2,h_1) + R_3^{8}(\zeta_1,\zeta_2,h_1,h_2) \\
      L_2^\prime	&= 0 \\
      h_1^\prime	&= b_c a_1^{-1/3} R_5^{9}(\zeta_1,\zeta_2;\Gamma_1^*,\Gamma_2^*) \\
      h_2^\prime	&= - b_c a_2^{-1/3} R_5^{9}(\zeta_1,\zeta_2;\Gamma_1^*,\Gamma_2^*) \\
      \Gamma_1^\prime&=\Gamma_2^\prime = x^\prime = y^\prime = 0,
    \end{aligned}
  \end{equation}	
  where $ R_j^{k} $ are degree $ k $ real-valued, homogeneous polynomials in $
  I_1,L_1,I_2,L_2 $ independent of the masses and $ R_5^9 $ is a degree $ 8 $
  polynomial in $ \Gamma_1^*,\Gamma_2^* $. $  R_1^4$ and $R_1^6 $ are given in Appendix
  \ref{sec:Constants}.
\end{proposition}
We omit the proof of the proposition as it is a huge computation that is far too
unwieldy to be contained here. The transform bringing the vector field $X$ to
the normal form $ X^9 $ and the resonant terms $ R_j^{8}, R_5^9 $ can be
provided upon request, the first few terms in powers of $L_i$ are listed in the appendix \eqref{eqn:R59}. 
Only the degree in $ \zeta_1,\zeta_2 $, and not the exact
form of the resonant terms, will be relevant to the arguments in the remainder
work.
		
There is a lot of information to unpack from Proposition
\ref{prop:normalformforSBC}. Firstly, the normal form procedure concludes with
the appearance of resonant terms at degree $ 9 $ in $ I_j,L_j $ for the $ h_j $
components. Consequently, the following corollary can be proved.
\begin{corollary}\label{cor:noFoliatation}
  It is not possible to construct a smooth foliation into invariant 4-planes
  normal to the simultaneous binary collision manifold $ \sbc $. Specifically,
  one can not find smooth invariants diffeomorphic to $ h_i $ at order $ 8 $ in
  $ I_i,L_i $.
\end{corollary}
\begin{proof}
  The result is an immediate consequence of
  \cite[Corollary~4.3]{duignanC83regularisationSimultaneous2020}. As the collinear problem sits inside
  the planar problem, and the impossibility of the foliation has been shown for
  the collinear problem, so too it must hold for the planar case.
\end{proof}
		
\begin{remark}
  Note that, after blow-up, degree $ 9 $ terms become degree $ 8 $ terms due to
  the rescaling by $ d\tau = r dt $. Therefore, an obstacle to the foliation is
  occurring at degree 8 in the blow-up space. This agrees with the results on
  the collinear problem \cite{duignanC83regularisationSimultaneous2020} where the degree $8$ was linked
  to the $ C^{8/3} $-regularity of the collinear problem.
\end{remark}

\begin{remark}\label{rem:integrals}
  Up to order 8, the normal form procedure has computed invariants given by the
  transformed $L_1,L_2,x,y,\Gamma_1,\Gamma_2$ and the additional
  \begin{equation}
    H = a_2^{-1/3} h_1 + a_1^{-1/3} h_2.
  \end{equation}
  There is in fact another one. The quantity $$ (I_1^3 + 3 I_1 L_1^2)
  - (I_2^3 + 3 I_2 L_2^2)$$ was used in Section \ref{sec:C0regularity} to
  obtain the integral $\kappa_2$ of the line field on the collision
  manifold. The quantity is an integral of the leading order dynamics. 
  It can be extended to an integral of $X^9$,
  \begin{equation}\label{eqn:kappaOriginal}
    \kappa = (I_1^3 + 3 I_1 L_1^2) - (I_2^3 + 3 I_2 L_2^2) + G^5(\zeta_1,\zeta_2,h_1,h_2) + G^7(\zeta_1,\zeta_2,h_1,h_2)
  \end{equation}
  where $G^j$ is a homogeneous degree $j$ polynomial in $\zeta_1,\zeta_2$. The
  full expressions are given in \eqref{eq:kappaFunctions}.
 % \comment{what about the $R^4$ and $R^8$ terms? In the collinear case we just had $R^6$...}
  The existence of each of these integrals will play a central role in showing
  the $R_4$ and $ R_6 $ terms do not affect the $ 8/3 $ regularity of the block map.
\end{remark}

\begin{remark}
The rectangular and the Caledonian problem mentioned in remark~\ref{rem:cal} 
are invariant sub-problems that are contained in the collision manifold.
For these it is known \cite{Elbialy1993planar} that a smooth regularisation is possible.
As described in remark~\ref{rem:cal} the condition that defines the collision 
manifold is $L_1 = L_2 = 0$ and $I_1 = I_2$. When $L_1 = L_2 = 0$, $\Gamma_i = 1$ the problem 
reduces to the collinear problem, and thus $R_5^9$ also reduces to the 
polynomial found in the collinear problem \cite{duignanC83regularisationSimultaneous2020} up to an overall factor that depends on $\Gamma_i$, see equation \eqref{eqn:R59} in the appendix. This polynomial vanishes when $I_1 = I_2$ and hence the term that is responsible for the $C^{8/3}$ regularisability vanishes.
\end{remark}

% \begin{remark}\label{rmk:canIgnoreR6}
%    The resonant term $ R_h $ appearing in the intrinsic energy components have
%			 $ b_c $ as a factor. Consequently, these terms come from the first
%			 coupling polynomial $ W_c $ in the expanded potential $ K $ given in
%			 \eqref{eqn:potential}. Moreover, the absence of $ b_{ij}, b_0 $ in the
%			 normal form show $ K_1, K_2 $ make no contribution to the resonant terms
%			 in the normal form. Lastly, the only terms in $ X_\R $ which are
%			 independent of the mass constants are the terms in the $ \zeta_i $
%			 components. As $ R_j^6 $ are independent of the masses as well, they must
%			 then come from the kinetic energy terms. A Hamiltonian with only kinetic
%			 terms has no singularities and is thus analytically regularisable. Under
%			 this reasoning, it must follow that the $ R_ $ terms make no contribution
%			 to any finite differentiability of the block map $ \pib $. This will be
%			 explicitly confirmed in the computation of the asymptotic series of $
%			 \pib $ in Section \ref{sec:proofOfMainTheorem}.
%		\end{remark}

\subsection{Geometric Sketch of Proof}\label{sec:GeometricSketch}
A procedure for determining the finite differentiability of the block map can
now be sketched. The procedure is not too dissimilar to the one developed in
showing the $ C^{8/3} $-regularity of the collinear problem \cite{duignanC83regularisationSimultaneous2020}. The crucial
difference is the higher co-dimensionality of $ \sbc $ in the planar problem in
comparison to the collinear problem. Consequently, the geometrical sketch is
harder to picture and more importantly, certain computations in the proof become
substantially more involved and require some new theory.
				
Recall that Proposition \ref{prop:structureOfCollisionManifold} gives the
topological structure of the flow near $ \sbc $. Understanding this structure
was instrumental in proving the $ C^0 $-regularity of the block map $
\pib:\Sigma_0 \to \Sigma_3 $. This structure was revealed by blowing up $ \sbc $
to produce the collision manifold $ \mathcal{C} $ and a consequent analysis of
the flow on $ \mathcal{C} $. The flow on $ \mathcal{C} $ will again be central
to determining the finite differentiability.
		
Any tubular neighbourhood $ \mathcal{T} $ of $ \mathcal{C} $ has a natural
decomposition due to Proposition \ref{prop:structureOfCollisionManifold}.
Overlapping neighbourhoods $ \mathcal{U} $, a tubular neighbourhood of the
normally hyperbolic invariant manifold $ \mathcal{N} $, and $ \mathcal{V} $, a
tubular neighbourhood of the homoclinic connection of $ \mathcal{N} $, can be
chosen so that $ \mathcal{T} = \mathcal{U}\cup\mathcal{V} $. The decomposition
splits $ \mathcal{T} $ into a region $ \mathcal{U} $ where the flow is
topologically equivalent to a neighbourhood of a manifold consisting entirely of
saddle singularities and a region $ \mathcal{V} $ where the flow is
topologically equivalent to a regular flow.
		
By splitting the flow into different topological regions, a geometric sketch for
computing the block map $ \pib:\Sigma_0\to \Sigma_3 $ unfolds. The idea is to
introduce an intermediate section, $ \Sigma_{int} \subset
\mathcal{U}\cap\mathcal{V} $, which is homeomorphic to $ (-a,a)\times
S^2\times\R^8 $ for $ a\ll 1 $, transverse to the flow on $ \mathcal{C} $, and
the intersection $ \mathcal{C}\cap\Sigma_{int} $ is homeomorphic to $ S^2 $. An
example of such a $ \Sigma_{int} $ is to take the boundary of the box $ [-1,1]^3
$ on $ \mathcal{C} $ in the coordinates $
(\beta_\alpha,\gamma_\alpha,\delta_\alpha) $ and take the direct product with an
interval $ (-a,a) $ in $ r_\alpha $ and with $ \R^8 $ from the remaining
variables. See Figure \ref{fig:secplotPlanar} for a depiction of $ \Sigma_{int}
$.
		
The conditions on $ \Sigma_{int} $ force it to surround $ \mathcal{N} $ so that,
if $ \mathcal{U} $ is taken sufficiently small, any orbit in $ \mathcal{U} $
will intersect $ \Sigma_{int} $. Then, as shown in the proof of Theorem
\ref{thm:C0Regularisable}, these near collision orbits shadow an orbit on $
\mathcal{C} $. After leaving $ \mathcal{U} $, as all the orbits on $ \mathcal{C}
$ are necessarily homoclinic trajectories, the near collision orbits will
transition through $ \mathcal{V} $ and necessarily reenter $ \mathcal{U} $.
Ultimately, the orbits intersect $ \Sigma_{int} $ again. It follows that $
\Sigma_{int} $ splits $ \pib $ over the different topological regions $
\mathcal{U},\mathcal{V} $,
\begin{equation*}
  \begin{aligned}
    &\pib = D_2\circ T\circ D_1,\\
    \qquad D_1:\Sigma_0 \to \Sigma_{int},\qquad &T:\Sigma_{int} \to
    \Sigma_{int},\qquad D_2: \Sigma_{int} \to \Sigma_3.
  \end{aligned}
\end{equation*}
		
\begin{figure}[ht]
  \centering %
	\def\svgwidth{0.5\linewidth}
	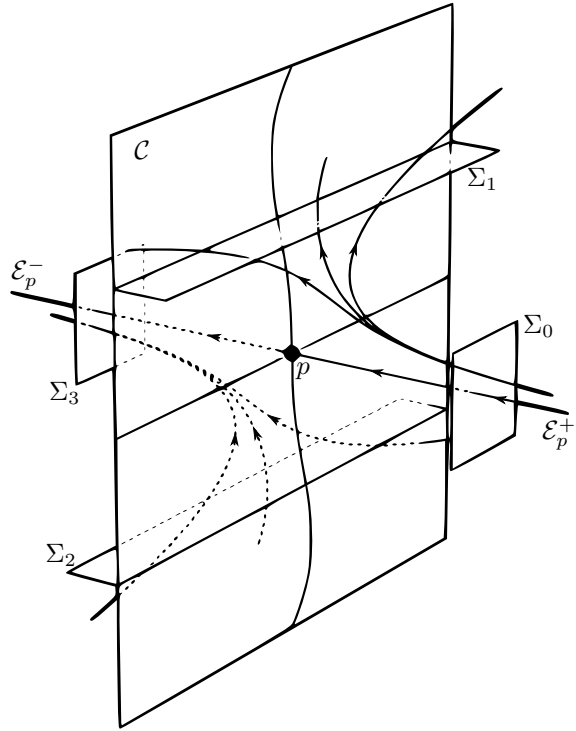

  \caption{A depiction of the hyperbolic transitions $ D_1, D_2 $. For clarity,
    only two of the walls, $ \Sigma_1, \Sigma_2 $, of $ \Sigma_{int} $ have been
    sketched. It is impossible to plot the co-dimension $ 4 $ manifold $
    \mathcal{N} $ of the planar problem. Instead, a co-dimension 3 problem has
    been plotted. Further, only a projection into a 3-dimensional slice, normal
    to $ \mathcal{N} $, containing $ p\in\mathcal{N} $, can be shown.}
  \label{fig:secplotPlanar}
\end{figure}
	
On the one hand the maps $ D_1, D_2 $ are transitions near a normally hyperbolic manifold of
saddle singularities. On the other hand, $T$ is a smooth transition map. The
asymptotic structure of the maps $ D_1,D_2 $ was studied in
\cite{duignanNormalFormsManifolds}. These asymptotic structures will be used in
the proceeding section to prove Theorem \ref{thm:C2Regularisation}. However,
before the details of the asymptotics are discussed, it is useful to make some
assumptions to get a clearer picture of how finite differentiability of the
block-map may creep in.
		
%		\comment{The variables $x,y,...$ have already been used otherwise}
In this sketch of the proof, the variables $x,y,z,w \in \R$ are used independently of any previous meaning.
It was shown in Proposition \ref{prop:structureOfCollisionManifold} that $ \NHIM
$ is a manifold of normally hyperbolic saddle singularities of co-dimension 4.
Assume that the one-dimensional invariant manifold leaving a point on $ \NHIM $
in the normal direction is stable. Then, up to scaling, the non-zero eigenvalues
of $ \NHIM $ are $ (-1,3,1,1) $. Assume that coordinates $
(x,y,z,w,u)\in\R^4\times\R^8 $ have been chosen near $ \NHIM $ so that the local
stable manifold is given by $ y=z=w=0 $ and the local unstable manifold by $ x =
0 $. To leading order in $ x,y,z,w,u $, the vector field in a tubular
neighbourhood of $ \NHIM $ is of the form,
\begin{equation}\label{eqn:linearisedNHIM}
  \begin{aligned}
    \dot{x}	&= x \\
    \dot{y}	&= -3y \\
  \end{aligned}
  \qquad
  \begin{aligned}
    \dot{z}	&= -z \\
    \dot{w}	&= -w \\
    \dot{u} &= 0
  \end{aligned}
\end{equation}
In general, system \eqref{eqn:linearisedNHIM} may have higher order terms. One
can ask whether such terms are removable, that is, whether the vector field near
$ \NHIM $ has a linearisation. This is a question of normal forms for $ \NHIM $.
This type of normal form was studied in \cite{duignanNormalFormsManifolds}.
There, it was shown that linearisation is only possible if the
non-zero-eigenvalues of $ \NHIM $, say $ \lambda\in \C^4 $, do not satisfy any
of the resonance conditions,
\[ n \cdot \lambda = \lambda_i, \qquad n \cdot \lambda = 0, \] where $ n\in\N^4
$ and $ n\cdot\lambda $ is the usual dot product. Unfortunately, for our
specific $ \NHIM $, the eigenvalues $ \lambda = (-1,3,1,1) $ do indeed satisfy
the resonance conditions, thus it cannot be assumed that the system admits a
linearisation. The exact form of the normal form is given in the appendix 
 as Proposition \ref{prop:normalFormNHIM}.
		
 Regardless, assume the vector field local to $ \NHIM $ can be linearised. If
 this is true, the system admits the integrals,
\begin{equation}\label{eqn:linearisedIntegrals}
  x^3 y, \quad x z, \quad x w.
\end{equation}
Now choose sections $ \Sigma_0 = \{ x = 1 \},\, \Sigma_3 = \{x=-1\} $ and take $
\Sigma_{int} $ as the boundary of the box $ (-a,a) \times [-1,1]^3\times \R^8 $.
Further, decompose $ \Sigma_{int} $ into its faces,
\[ \Sigma_\eta^\pm := \{\eta=\pm1\}\cap \Sigma_{int} \text{ for } \eta=y,z,w
  . \] In turn, $ D_1, D_2 $ can be decomposed into the directional transitions
\[ D_{1,\eta}^\pm : \Sigma_0 \to \Sigma_{\eta}^\pm,\qquad D_{2,\eta}^\pm :
  \Sigma_\eta^\pm \to \Sigma_3,\] with $ \eta = y,z,w $. Further, the domain $
\mathcal{D}_\eta^\pm $ of each $ D_{1,\eta}^\pm $ is a different subset of $
\Sigma_0 $. It follows that the block map $ \pib $ is split over these domains
into $ \pib_\eta^\pm $. 

% \comment{Added this paragraph to deal with technicalities of $T_\eta$}

Splitting $T$ over the different faces of $\Sigma_{int}$ requires some technicalities. It is entirely possible that trajectories on $\Sigma_\eta^+$ do not reach $\Sigma_\eta^-$, instead passing through another face of $\Sigma_{int}$. There are two ways to solve this issue; restrict the domain of $T|_{\Sigma_\eta^+}$ so that the codomain is $\Sigma_\eta^-$, or modify the faces of $\Sigma_{int}$ so that the codomain is contained in $\Sigma_{\eta}^-$. We will take the latter solution as asserting the domain rather than the codomain of $T$ is more useful for computing $\pib$. Specifically, we can modify $\Sigma_{int}$ so that a given face $\Sigma_\eta^-$ is sufficiently wide enough to capture all trajectories from $\Sigma_\eta^+$. Hence, we can split, $T_\eta:\Sigma_\eta^+ \to \Sigma_\eta^-$.
		
Now, let us first focus on $ \pib_{y}^+ $. If one takes coordinates $
(-y_3,-z_3,-w_3,u_3) $ on $ \Sigma_3 $ and coordinates $ (-x_2,-z_2,-w_2,u_2) $
on $ \Sigma_y^- $ then the integrals \eqref{eqn:linearisedIntegrals} give $
D_{2,y}^- $ as,
\begin{equation}
  D_{2,y}^-:\qquad	y_3 = x_2^3, \qquad z_3 = x_2 z_2, \qquad w_3 = x_2 w_2,\qquad u_3 = u_2.  
\end{equation}
Similarly, taking coordinates $ (y_0,z_0,w_0,u_0) $ on $ \Sigma_0 $ and
coordinates $ (x_1,z_1,w_1,u_1) $ on $ \Sigma_y^+ $, $ D_{1,y}^+ $ is obtained
as
\begin{equation}
  D_{1,y}^+:\qquad x_1 = y_0^{1/3}, \qquad z_1 = y_0^{-1/3} z_0,\qquad w_1 = y_1^{-1/3} w_0, \qquad u_1 = u_0.
\end{equation}
It is immediate that $ D_{1,y}^+ = (D_{2,y}^-)^{-1} $. However, it will be shown
in this section that this may not be true when higher order terms are
considered.
		
Note that, analysing solutions near-collision amounts to treating $ x_2 $ and $
y_0,z_0,w_0 $ as small variables. $ D_2 $ is clearly continuous in $ x_2 $,
however, $ D_{1,y}^+ $ is not continuous in $ y_0,z_0,w_0 $. This is a
consequence of the fact that $ D_1:\Sigma_0 \to \Sigma_{int} $ maps the 8
dimensional surface $ (0,0,0,u_0) $ to the $ 10 $ dimensional surface $
(0,z_1,w_1,u_1) $. However, if $ y_0,z_0,w_0 $ are rescaled by $ \e\ll 1 $ to
reflect the fact they are all small, then
\[ D_{1,y}:\qquad y_1 = \e^{1/3} y_0^{1/3}, \qquad z_1 = \e^{2/3} y_0^{-1/3}
  z_0,\qquad w_1 = \e^{2/3} y_1^{-1/3} w_0, \qquad u_1 = u_0, \] is continuous
in $ \e $ on $ [0,1) $ for any choice $ (y_0,z_0,w_0) $ with $ y_0 \neq 0 $.
This is sufficient to compute the asymptotic expansion of $ \pib $ as only an
asymptotic expansion in $ \e $ is desired.
		
Here is the crucial point of the sketch of finite differentiability. As $
D_{1,y}^+ = (D_{2,y}^-)^{-1} $ then clearly their composition $ D_{2,y}^- \circ
D_{1,y}^+ $ is the identity and hence smooth. However, when the smooth
transition map $ T $ is introduced, $ \pib_y^+ $ becomes an asymptotic series in
$ \e^{1/3} $. This is most easily conveyed by letting $ T $ be a smooth map of
the form,
\begin{equation}\label{eqn:suggestedT}
  T: (x_1,z_1,w_1,u_1) \mapsto (x_1,z_1,w_1,u_1 + a(z_1,w_1) x_1^8),
\end{equation} 
for some smooth function $ a $. Then composing yields,
\begin{align*}
  \pib_y^+	&= D_{2,y}^- \circ T \circ D_{1,y}^+ \\	
						&= (\e y_0,\e z_0, \e w_0, u_0 + a(0,0) \e^{8/3} y_0^{8/3} +\dots),
\end{align*}
giving the $ 8/3 $ regularity. This is the essence of the finite
differentiability in the regularisation of simultaneous binary collisions.
		
It is worth remarking that the block map $ \pib $ may be smoother in certain
directions. To see this, consider the component $ \pib_z^+ $. The hyperbolic
transition maps are given by
\begin{align*}
  D_{1,z}^+: (x_1,y_1,w_1,u_1) &= ( \e z_0,\ \e^{-2} z_0^{-3} y_0,\ z_0^{-1} w_0,\ u_0 ) \\
  D_{2,z}^-: (y_3,z_3,w_3,u_3) &= (x_2^3 y_2,\ x_2,\ x_2 w_2,\ u_2)
\end{align*}
Provided $ y_0\neq 0 $ we have, as $ \e \to 0 $, $ y_1 \to \infty $. Hence $
D_{1,z}^+ $ is not continuous in $ \e $ for $ y_0 \neq 0 $. But, $ y_1 \to
\infty $ merely implies, as we consider smaller $ \e $, that points starting on
$ \Sigma_0 $ will intersect a different face of $ \Sigma_{int} $. In this case,
provided $ y_0 \neq 0 $, points on $ \Sigma_0 $ for sufficiently small $ \e $
will intersect $ \Sigma_y^\pm $. The only case left to treat then is $ y_0 = 0
$.
		
Now assume the transition map $ T $ between $ \Sigma_y^+ $ and $ \Sigma_y^- $ is
something of the form
\[ T : (x_1,0,w_1,u_1) \mapsto (x_1,0,w_1,u_1 + b(w_1) x_1^8), \] as before with
$ b $ some smooth function. Then composing yields,
\[ \pib_z^+ = D_{2,z}^- \circ T \circ D_{1,z}^+ = (0,\e z_0, \e w_0, u_0 + b(w_0
  z_0^{-1}) \e^{8} z_0^{8} +\dots). \] Provided $ z_0 \neq 0 $, the regularity
in the $ z $- direction is then $ 8 $, smoother than the $ 8/3 $ of the $ y
$-direction. The final direction, the $ w $-direction, is analogous to the $ z
$-direction.
		
To turn this sketch into a proof, one has to deal with all the possible resonant
terms in the normal form which prevent the linearisation of the vector field
near $ \NHIM $. Moreover, the transition map $ T $ must be shown to behave as
suggested in \eqref{eqn:suggestedT}.

\subsection{Proof of $ C^{8/3} $-Regularisation}
	
In this section the necessary theory to prove Theorem \ref{thm:C2Regularisation}
on the $ C^{8/3} $-regularisation of simultaneous binary collisions is developed.
The work in \cite{duignanNormalFormsManifolds} is used in Appendix
\ref{sec:proofOfLemma} to obtain Lemma \ref{lem:ApproxDulacMap}, an asymptotic
expansion of the hyperbolic transitions $ D_1, D_2 $ in the planar 4-body
problem. Then the smooth transition map $ T $ is computed to sufficiently high order using the integrals found in Remark \ref{rem:integrals} 
and the result summarised in Lemma \ref{lem:ApproxT}. Ultimately, the
asymptotic series computed for each $ D_2,T,D_1 $ are composed to prove Theorem
\ref{thm:C2Regularisation}.
		
\subsubsection{Asymptotic Structure of the Hyperbolic Transitions}

Following the geometric sketch in Section \ref{sec:GeometricSketch}, we seek the
asymptotic expansions of the transition $ D_1,D_2 $ near the normally hyperbolic
manifold $ \NHIM $ corresponding to simultaneous binary collisions. Consider
again the system $ X^9 $, defined in \eqref{prop:normalformforSBC}. We want to
choose coordinates so that $ \NHIM $ corresponds to $
[\alpha,\beta,\gamma,\delta] = [1,0,0,0] $. In doing so, it can be ensured that the
chart $ U_\alpha $ contains $ \NHIM $ whilst $ U_\beta,U_\gamma,U_\delta $
contain no singularities. Similar to the collinear problem in
\cite{duignanC83regularisationSimultaneous2020}, this is done by first making the transform,
\begin{equation}
  J_1 = \frac{1}{2}(I_1 + I_2),\qquad J_2 = \frac{1}{2}(I_1 - I_2)
\end{equation}
Then $ X^9 $ is transformed to a system of the form,
\begin{equation}\label{eqn:rotNormalForm}
  \begin{aligned}
    J_1^\prime	&= J_1^2+J_2^2+\frac{1}{2}(L_1^2+L_2^2) + \tilde{R}_1^{4}(J_j,L_j,h_1,h_2)+ \tilde{R}_1^{6}(J_j,L_j,h_1,h_2) + \tilde{R}_1^{8}(J_j,L_j,h_1,h_2) \\
    L_1^\prime	&= 0 \\
    J_2^\prime	&= -2 J_1 J_2 - \frac{1}{2}(L_1^2-L_2^2)+\tilde{R}_3^{4}(J_j,L_j,h_1,h_2)+ \tilde{R}_3^{6}(J_j,L_j,h_1,h_2) + \tilde{R}_3^{8}(J_j,L_j,h_1,h_2) \\
    L_2^\prime	&= 0 \\
    h_1^\prime	&= a_1^{-1/3} b_c \tilde{R}_5^{9}(J_j,L_j; \Gamma_j^*) \\
    h_2^\prime	&= -a_2^{-1/3} b_c \tilde{R}_5^{9}(J_j,L_j;\Gamma_j^*) \\
    \Gamma_1^\prime&=\Gamma_2^\prime = x^\prime = y^\prime = 0.
  \end{aligned}
\end{equation}	
Let $ \tilde{X}_9 $ be the vector field associated to \eqref{eqn:rotNormalForm}.
			
Performing the $ \alpha $-directional blow-up on $ \tilde{X}_9 $, we obtain the
system,
\begin{equation}\label{eqn:Xtilalpha}
  \begin{aligned}
    \ra^\prime	&= \ra (1 + \ga^2 +\frac{1}{2}(\ba^2 +\da^2) )+ \ra^3\tilde{R}_1^{4}(1,\ba,\ga,\da,h) + \ra^5\tilde{R}_1^{6}(1,\ba,\ga,\da,h) + O(\ra^7)  \\
    \ba^\prime	&= -\ba \ra^{-1} \ra^\prime \\
    \ga^\prime	&= -3 \ga -\ga\left(\ga^2 +\frac{1}{2}(\ba^2 +\da^2)\right) - \frac{1}{2}(\ba^2-\da^2) + \ra^3\bar{R}_{3,\a}^4(\ba,\ga,\da,h) + \ra^5\bar{R}_{3,\a}^{6}(\ba,\ga,\da,h) + O(\ra^7) \\
    \da^\prime	&= -\da \ra^{-1} \ra^\prime \\
    h_1^\prime	&= a_1^{-1/3} b_c \ra^8 \tilde{R}_5^{9}(1,\ba,\ga,\da) +O(\ra^9) \\
    h_2^\prime	&= -a_2^{-1/3} b_c \ra^8 \tilde{R}_5^{9}(1,\ba,\ga,\da) + O(\ra^9) \\
    x^\prime &= y^\prime = \Gamma_1^\prime = \Gamma_2^\prime = 0 +O(\ra^9)
  \end{aligned}
\end{equation}
with
\begin{equation}
  \bar{R}_{3,\a}^j (\ba,\ga,\da,h) := \tilde{R}_3^j(1,\ba,\ga,\da,h) - \ga \tilde{R}_1^j(1,\ba,\ga,\da,h).
\end{equation}
Let the vector field associated to \eqref{eqn:Xtilalpha} be $ \tilde{X}_\alpha$.
			
% Using the integrals computed in \ref{sec:untwistingIntegrals} the lower order
% terms in $ \tilde{X}_\a $ can be removed. Firstly, make the time rescaling $
% d\tilde{\tau} = (1 + \ga^2 +\frac{1}{2}(\ba^2 +\da^2) ) d\tau $ and introduce
% the new coordinate $ \tga $ via the near identity transformation,
% \begin{equation}\label{eqn:tga}
%   \tga = \ga+\frac{1}{3}\ga^3 + \frac{1}{2} \da^2(\ga-1) + \frac{1}{2} \ba^2 (1+\ga) .
% \end{equation}
% Note that $ \tga $ is simply the integral $ \kappa_2 $, defined in Proposition
% \ref{prop:CollisionManifoldInvariants}, in the current coordinates. By
% performing the time rescaling and introducing of $ \tga $ as a coordinate, $
% \tilde{X}_\a $ is transformed to,

Using the approximate integral $\kappa$ introduced in \eqref{eqn:kappaOriginal},
the leading order terms of the normal form transformation for the hyperbolic
singularity at $\NHIM$ can be computed. In the current coordinates
$r_\alpha,\ba,\ga,\da$ the approximate integral is given by,
\begin{equation}
    \kappa_\alpha = 6\ra^3\left(\ga+\frac{1}{3}\ga^3 + \frac{1}{2} \da^2(\ga-1) +
      \frac{1}{2} \ba^2 (1+\ga)\right) + \ra^5 \tilde{G}_\alpha^5(\ba,\ga,\da) +
      \ra^7\tilde{G}_\alpha^7(\ba,\ga,\da) + O(\ra^{10}). 
\end{equation}
with $\tilde{G}^j(J_1,L_1,J_2,L_2) = G^j(J_1+J_2,L_1,J_1-J_2,L_2)$ and $\tilde{G}^j_\alpha$ is $\tilde{G}^j$ in the $U_\alpha$ chart.

In the vector field the equations for $\ra, \ba, \da$ are all proportional to $\ra'/\ra$, 
which can be removed by a time scaling. The equation for $\ga$ can be simplified 
by observing that up to a factor $6 \ra^3$ the integral to leading order is equal 
to $\ga$. Thus it is used to introduce a modified $\tga$ which satisfies a much 
simpler equation.
Thus introducing the coordinate $\tga$ through 
\begin{equation}
  \label{eq:normalformtraf}
  \tga = \tfrac16 \ra^{-3}\kappa,
\end{equation}
provides the normal form coordinates to order $8$ in $\ra$. With an additional
time rescaling of $d\tau = \ra^{-1}\ra^\prime d\tilde{\tau}$, the vector field
\eqref{eqn:Xtilalpha} is put in the normal form
\begin{equation}\label{eqn:ApproxNF}
  \begin{aligned}
    \ra^\prime	&= \ra  + O(\ra^8)  \\
    \ba^\prime	&= -\ba + O(\ra^8)\\
    \tga^\prime	&= -3 \tga + O(\ra^8) \\
    \da^\prime	&= -\da +  O(\ra^8)\\
    h_1^\prime &= h_2^\prime = x^\prime = y^\prime = \Gamma_1^\prime =
    \Gamma_2^\prime = 0 + O(\ra^8)
  \end{aligned}
\end{equation}
an all the extraneous polynomial terms of degree 4 and 6 have been removed.
			
As argued in Section \ref{sec:GeometricSketch}, we want to compute the
transitions $ D_1:\Sigma_0\to\Sigma_{int} $ and $ D_2:\Sigma_{int}\to\Sigma_3 $
for some surfaces $ \Sigma_0,\Sigma_3 $ transverse to the manifolds $
\mathcal{E}^+,\mathcal{E}^-$ respectively, and $ \Sigma_{int} $ homeomorphic to
$ (-a,a)\times S^2\times \R^8 $ and transverse to the unstable manifold. We will
study these transitions for the specific sections,
% \comment{I don't understand the definition of $\Xi_{int}$. In the sketch it seemed different?}
\begin{equation}
  \begin{aligned}
    \Sigma_0	&= \{ (1,{\b}_0,{\g}_0,{\d}_0,u_0) | ({\b}_0,{\g}_0,{\d}_0) \in [-1,1]^3, u_0 \in \R^8 \},\\
    \Sigma_{int} &= (-a,a) \times \partial([-1,1]^3)\times \R^8,  \\
    \Sigma_3 &= \{ (-1,-{\b}_3,-{\g}_3,-{\d}_3,u_3) | ({\b}_3,{\g}_3,{\d}_3) \in
    [-1,1]^3, u_3 \in \R^8 \}.
  \end{aligned}
\end{equation}
The opposite orientation of $ \Sigma_3 $ has been chosen to match the
non-orientability of $ \mathcal{C}\cong\RP^3\times\mathcal{M} $. The maps $
D_1,D_2 $ between these two specific sections will be referred to as
\textit{Dulac maps} due the analogous properties of the well studied Dulac maps
near hyperbolic saddles in $ \R^2 $. A good review of results known for the
planar Dulac maps is given in \cite{roussarie1995bifurcations}. Theory for Dulac
maps near normally hyperbolic invariant manifolds is given in
\cite{duignanNormalFormsManifolds}.
			
Studying the hyperbolic transitions $ D_1,D_2 $ is most easily done by splitting
up their actions on the various faces of $ \Sigma_{int} $. Explicitly, let
\[ \Sigma_\eta^\pm := \{\eta=\pm1\}\cap \Sigma_{int} \text{ for }
  \eta=\ba,\tga,\da , \] so that $ \Sigma_{int} = \bigcup_\eta \Sigma_{\eta}^\pm
$ and define,
\[ D_{1,\eta}^\pm : \Sigma_0 \to \Sigma_{\eta}^\pm,\qquad D_{2,\eta}^\pm :
  \Sigma_\eta^\pm \to \Sigma_3. \] For each $ \eta = \ba,\tga,\da $, the maps $
D_{2,\eta}^\pm $ will be denoted the \textit{$ \eta $-directional Dulac map} and
$ D_{1,\eta}^\pm $ its inverse.

% \comment{explain why it is sufficient...}
To prove the $ C^{8/3} $-regularity of the block map $ \pib $ it is
sufficient to study just the maps $ D_{1,\tga}^+,D_{2,\tga}^- $,
because the factor 3 in the vector field only appears in the $\tga$ component.
Let $(r_1,\b_1,\d_1,u_1), (-r_2,-\b_2,-\d_2,u_2), $ be coordinates on $
\Sigma_{\tga}^+,\Sigma_{\tga}^- $ respectively. From both Proposition
\ref{prop:DirectionalDulacMaps} and Lemma \ref{prop:invDOrder} the following
lemma on the asymptotic expansion of the hyperbolic transitions near $ \NHIM $
can be concluded.
\begin{lemma}\label{lem:ApproxDulacMap}
  Scale $ (\b_0,\g_0,\d_0) = \e (\b_0,\g_0,\d_0) $. Then the hyperbolic
  transitions $ D_{1,\tga}^+ $ and $ D_{2,\tga}^- $ have the asymptotic
  approximations
  \begin{equation}
    \begin{aligned}
      D_{1,\tga}^+ &:\quad
      \begin{aligned}
        {r}_1	&\sim \e^{1/3} {\g}_0^{1/3} \left(1 + O(\e^{3}\ln\e) \right) \\
        {\b}_1 &\sim \e^{2/3} {\g}_0^{-1/3} {\b}_0\left(1 +
          O(\e^{3}\ln\e)\right)
      \end{aligned}\quad
      \begin{aligned}
        {\d}_1	&\sim \e^{2/3} {\g}_0^{-1/3} {\d}_0 \left(1 + O(\e^{3}\ln\e) \right)\\
        u_1 &\sim u_0 + O(\e^{3}\ln\e)
      \end{aligned}\\[1em]
      D_{2,\tga}^- &:\quad
      \begin{aligned}
        {\b}_3		&\sim  {r}_2 {\b}_2 \left(1 + O(\e^{3}\ln\e)\right)\\
        {\g}_3 &\sim {r}_2^{3} \left(1 + O(\e^{3}\ln\e) \right)
      \end{aligned}\quad
      \begin{aligned}
        {\d}_3	&\sim  {r}_2 {\d}_2 \left(1 + O(\e^{3}\ln\e) \right)\\
        u_3 &\sim u_2 + O(\e^{3}\ln\e)
      \end{aligned}
    \end{aligned}
  \end{equation}
  for $ u = (h_1,h_2,x,y,\Gamma_1,\Gamma_2) $.
\end{lemma}
There are several technical details needed to prove Lemma~\ref{lem:ApproxDulacMap}. 
The proof is given in Appendix~\ref{sec:proofOfLemma}.
		
\subsubsection{Smooth Transition $ T $}

In this section we show that $ T $ is of the form sketched in
\eqref{eqn:suggestedT}. As mentioned in Section \ref{sec:GeometricSketch}, it is
useful to split $ \Sigma_{int} $ into its various faces $ \Sigma_{\eta}^\pm,
\, \eta = y,z,w $. In doing so, $ T $ should also split over its action on the
various faces. Moreover, recall that, due to the fact that $ \mathcal{C} $ is a
M\"{o}bius bundle, $ T: \Sigma_{\eta}^\pm \to \Sigma_{\eta}^\mp $. For clarity,
denote each restriction by,
\[ T_\eta^\pm: \Sigma_{\eta}^\pm \to \Sigma_{\eta}^\mp,\quad \eta = \b,\d,\g. \]
			
We will focus on computing $ T_{\tga}: \Sigma_{\tga}^+ \to \Sigma_{\tga}^- $.
This can be computed from the $ \gamma $-directional blow-up. In the rotated
system $ \tilde{X}_9 $, given by \eqref{eqn:rotNormalForm}, the $ \gamma
$-directional blow-up produces the vector field $ \tilde{X}_\g $, given as,
\begin{equation}\label{eqn:Xtilgamma}
  \begin{aligned}
    \rg^\prime	&= -\rg(2\ag + \frac{1}{2}(\bg^2-\dg^2)) +\rg^3 \tilde{R}_3^4(\ag,\bg,1,\dg,h)+ \rg^5\tilde{R}_3^{6}(\ag,\bg,1,\dg,h) + O(\rg^7)  \\
    \ag^\prime	&=  1 + 3\ag^2 + \frac{1}{2}(\bg^2+\dg^2)+ \frac{1}{2}\ag(\bg^2-\dg^2)+ \rg^3\bar{R}_{1,\g}^{4}(\ag,\bg,\dg,h) + \rg^5\bar{R}_{1,\g}^{6}(\ag,\bg,\dg,h) + O(\rg^7)\\
    \bg^\prime	&= -\bg\rg^{-1}\rg^\prime + O(\rg^7) \\
    \dg^\prime	&= -\dg\rg^{-1}\rg^\prime + O(\rg^7)\\
    h_1^\prime	&= b_c \rg^8 \tilde{R}_5^{9}(\ag,\bg,1,\dg) +O(\rg^9) \\
    h_2^\prime	&= -b_c \rg^8 \tilde{R}_5^{9}(\ag,\bg,1,\dg,) + O(\rg^9) \\
    x^\prime &= y^\prime = \Gamma_1^\prime = \Gamma_2^\prime = 0 + O(\rg^9)
  \end{aligned}
\end{equation}
with
\begin{equation}
  \bar{R}_{1,\g}^j (\ag,\bg,\dg,h) := \tilde{R}_1^j(\ag,\bg,1,\dg,h) - \ag \tilde{R}_3^j(\ag,\bg,1,\dg,h).
\end{equation}

To aid in computing $T_\tga$, we will first use the integral $\kappa$ to
simplify \eqref{eqn:Xtilgamma}. In the $U_\gamma$ chart we have $\kappa$ given by,
\begin{equation}
    \kappa_\gamma = \rg \left((2 + 6 \ag^2 + 3 \bg^2 + 3 \dg^2 + 3 \ag \left(\bg^2 - \dg^2\right) \right)+ \ra^2 \tilde{G}^5_\gamma(\ag,\bg,\dg) +
      \ra^4\tilde{G}^7_\gamma(\ag,\bg,\dg) + O(\ra^{7}),
\end{equation}
where $ \tilde{G}^i_\gamma $ is $\tilde{G}^i$ in the $U_\gamma$ chart. 

Introduce the coordinate $\trg$ through 
\begin{equation}
        \trg = \kappa_\gamma^{1/3}.
\end{equation}
Then $\trg$ will satisfy,
\begin{equation}\label{eqn:trgDiffEq}
    \trg^\prime = 0 + O(\trg^7).
\end{equation}

Using $\trg$ and the approximate integrals $L_1 = \rg \bg, L_2 = \rg \dg$ we are able to prove the following lemma giving the transition map $T_{\tga}$.
\begin{lemma}\label{lem:ApproxT}
  Take $ u=h_1,h_2,x,y,\Gamma_1,\Gamma_2 $ and let $
  ({\ra}_1,{\ba}_1,{\da}_1,u_1),\ (-{\ra}_2,-{\ba}_2,-{\da}_2,u_2) $ be
  coordinates on $ \Sigma_{\tga}^+, \Sigma_{\tga}^- $ respectively. Then the
  transition map $ T_{\tga}:\Sigma_{\tga}^+ \to \Sigma_{\tga}^- $ is,
  \begin{equation}\label{eqn:Tina}
    T_{\tga} : \quad
    \begin{aligned}
      {\ra}_2  		&= {\ra}_1 + O({\ra}_1^7) \\
      {\ba}_2 &= {\ba}_1 + O({\ra}_1^7)
    \end{aligned}
    \quad 
    \begin{aligned}
      {\da}_2		&= {\da}_1 + O({\ra}_1^7) \\
      u_2 &= u_1 + O({\ra}_1^8).
    \end{aligned}
  \end{equation}
\end{lemma}
\begin{proof}
  The transition $T_\tga$ in the lemma is given in the coordinates $\eta_\alpha$. First we compute the transition in the $\eta_\gamma$ coordinates before converting to the $\eta_\alpha$ coordinates. Let $\tilde{\Sigma}_\tga^+,\tilde{\Sigma}_\tga^+$ be the respective images of $\Sigma_\tga^+,\tilde{\Sigma}_\tga^-$ in the $U_\gamma$ chart and let $({\rg}_1,{\bg}_1,{\dg}_1,u_1)$, $(-{\rg}_2,-{\bg}_2,-{\dg}_2,u_2) $ be respective coordinates on these sections. 
  
  From \eqref{eqn:trgDiffEq} we can conclude that,
  \[ \tilde{r}_{\gamma 2} = \tilde{r}_{\gamma 1} + O({\rg}_1^7) \implies \kappa_1^{1/3} = \kappa_2^{1/3} + O({\rg}_1^7), \]
  where $\tilde{r}_{\gamma i},\kappa_i$ are considered as functions of $ ({\rg}_i,{\bg}_i,{\dg}_i) $. From \eqref{eq:normalformtraf} we have that,
  \[ \kappa_2^{1/3} = \kappa_1^{1/3} + O({\rg}_1^7) \implies {\ra}_2 \tilde{\gamma}_{\alpha 2}^{1/3} = -{\ra}_1 \tilde{\gamma}_{\alpha 1}^{1/3} + O({\ra}_1^7) \implies {\ra}_2 = {\ra}_1 + O({\ra}_1^7),\]
  with the last implication a result of the fact that $\tilde{\gamma}_{\alpha 1} = 1,\tilde{\gamma}_{\alpha 2} = -1$ are the definitions of the sections $\Sigma_{\tga}^+,\Sigma_{\tga}^-$, respectively.
  
  Similarly, by using the fact that $L_1 = r_{\gamma 1} \beta_{\gamma 1}, L_2 = r_{\gamma 2} \beta_{\gamma 2}$ are approximate integrals to order $\rg^7$, we find that,
  \[ {\ba}_2 = {\ba}_1 + O({\ra}_1^7),\qquad {\da}_2 = {\da}_1 + O({\ra}_1^7).\]
  The equations for $u_i$ are also immediate from their approximate invariance in \eqref{eqn:Xtilgamma}.
\end{proof}
		
\subsubsection{Proof of Theorem \ref{thm:C2Regularisation}}
		
Now that the directional Dulac maps and the smooth transition have been
approximated in lemmas \ref{lem:ApproxDulacMap} and \ref{lem:ApproxT}
respectively, the following main theorem can be proved.
\begin{thm}\label{thm:C2Regularisation}
  For any choice of masses, the simultaneous binary collision is precisely $
  C^{8/3} $-regularisable in the planar $ 4 $-body problem.
\end{thm}
\begin{proof}
  It must be shown that the block map $\pib$ is at least, and at most, $C^{8/3}$
  regular. In \cite{duignanC83regularisationSimultaneous2020} it was shown that the collinear problem is
  precisely $C^{8/3}$ regularisable. The collinear problem naturally embeds
  inside the planar problem. Hence, there is a submanifold, the collinear
  submanifold, for which the restriction of $\pib$ is known to be $C^{8/3}$
  regularisable. This automatically gives that the full map $\pib$ is at best
  $C^{8/3}$. However, the question remains whether it is of worse regularity
  away from collinearity.

  In what follows, let $ u = (h_1,h_2,x,y,\Gamma_1,\Gamma_2) $. We will collect
  the various results of the preceding section for the proof.
				
  Take the two sections,
  \[ \Sigma_{0} = (1,{\ba}_0,{\ga}_0,{\da}_0,u_0), \qquad \Sigma_{3} =
    (-1,-{\ba}_3,-{\ga}_3,-{\da}_3,u_3) \] defined in the $ \alpha $-directional
  blow-up. Note that $ \Sigma_0,\Sigma_3 $ is transverse to $ \mathcal{E}^+ $
  and $ \mathcal{E}^- $ respectively. Further, define $ \tga $ through equation
  \eqref{eq:normalformtraf} and consider the two sections, transverse to $ \mathcal{C} =
  \{\ra = 0 \} $,
  \[ \Sigma_{\tga}^+ = \{({\ra}_1,{\ba}_1,{\tga},{\da}_1,u_1) | \tga = 1 \},
    \qquad \Sigma_{\tga}^- = \{(-{\ra}_2,-{\ba}_2,{\tga},-{\da}_2,u_2) | \tga =
    -1 \}. \]
				
  The section of the block map with $ \tga > 0 $ is then given by,
  \begin{align*}
    \pib_{\g}^+ &= D_{2,\tga}^- \circ T_{\tga} \circ D_{1,\tga}^+,\\
    D_{1,\tga}^+:\Sigma_{0}\to\Sigma_{\tga}^+ ,\quad &T_{\tga}:\Sigma_{\tga}^+
    \to \Sigma_{\tga}^-,\quad D_{2,\tga}^-: \Sigma_{\tga}^-\to \Sigma_{3}. 
  \end{align*}
  Lemma \ref{lem:ApproxDulacMap} gives the asymptotic expansion of $
  D_{2,\tga}^-, D_{1,\tga}^+ $ whilst Lemma \ref{lem:ApproxT} gives the
  asymptotic expansion of $ T_{\tga} $. Composing the maps and keeping track of
  the known error, we have that,
  \begin{align*}
    \pib_{\g}^+ &= D_{2,\tga}^- \circ T_{\tga} \circ D_{1,\tga}^+(\e{\ba}_0,\e{\ga}_0,\e{\da}_0,u_0) \\
                &= D_{2,\tga}^- \circ T_{\tga}\left(
                  \e^{1/3} {\g}_0^{1/3} \left(1 + O(\e^{3}\ln\e) \right),\ 
                  \e^{2/3} {\g}_0^{-1/3} {\b}_0\left(1 + O(\e^{3}\ln\e)\right), \right.\\
                &\qquad\qquad \left.
                  \e^{2/3} {\g}_0^{-1/3} {\d}_0 \left(1 + O(\e^{3}\ln\e) \right),\
                  u_0 + O(\e^{3}\ln\e)
                  \right) \\
                &= D_{2,\tga}^- \left(
                  \e^{1/3} {\g}_0^{1/3} \left(1 + O(\e^{3}\ln\e) + O(\e^{8/3}) \right),\ 
                  \e^{2/3} {\g}_0^{-1/3} {\b}_0\left(1 + O(\e^{2}\ln\e) + O(\e^{8/3})\right), \right.\\
                &\qquad\qquad \left.
                  \e^{2/3} {\g}_0^{-1/3} {\d}_0 \left(1 + O(\e^{3}\ln\e) + O(\e^{7/3}) \right),\
                  u_0 + O(\e^{3}\ln\e) + O(\e^{8/3})
                  \right) \\
                &= \left( \e {\ba}_0 + O(\e^{3}),\ \e {\ga}_0 + O(\e^{3}),\ \e {\da}_0 + O(\e^{3}),\ {u}_0 + O(\e^{8/3}) \right).
  \end{align*}
  Now, the flatness $ O(\epsilon^{8/3}) $ implies the flatness of the same order
  in the variables. Hence $ \pib $ is at least $ C^{8/3} $ on $ \ga > 0 $ with 
  second derivatives exactly $ 0 $.
				
  To finish the proof that $\pib$ is at least $C^{8/3}$ we note that by swapping
  signs in the calculations of the preceding section, one can quickly identify
  that $ \pib $ is $ C^{8/3} $ on $ \tga < 0 $ also, again with the second
  derivative vanishing. Finally, from Theorem \ref{thm:C0Regularisable} the block
  map $ \pib $ is continuous. It follows that, if the second derivative is $ 0 $
  in all directions $ \ga \neq 0 $, then so it must also be on $ \ga = 0 $.
  Hence $ \pib $ is at least $ C^{8/3} $ even away from collinearity.
\end{proof}

\section{Concluding Remarks}
  If we keep track of the irremovable terms $R_5$ in the $h_j$ components of the
  normal form $X_9$ then ultimately they produce non-zero terms at
  order $\e^{8/3}$ of the $h_j$ components of the block map $\pib$. That is
  $h_{j3}$ has the form,
  \begin{equation}
    h_{j3} = h_{j0} + F_j({\ba}_0,{\ga}_0,{\da}_0; \Gamma_1^*,\Gamma_2^*) \e^{8/3} + O(\e^3)
  \end{equation}
  with $F_j$ a $C^{8/3}$ function in ${\ba}_0,{\ga}_0,{\da}_0$ and polynomial in
  $\Gamma^*_1,\Gamma_2^*$. If the function $F_j$ is zero for a given set of
  values $({\ba}_0^*,{\ga}_0^*,{\da}_0^*,\Gamma_1^*,\Gamma_2^*)$ the block map $\pib$
  will be more than $C^{8/3}$ regular in the direction
  $[{\ba}_0^*,{\ga}_0^*,{\da}_0^*]$ at the point $\Gamma_1^*,\Gamma_2^*$ along
  $\NHIM$.
  
  For the collinear problem, $F_j$ has been explicitly computed  \cite{duignanC83regularisationSimultaneous2020}.
  This was done by blowing up and explicitly computing the smooth transition map to order $8$ (rather than order 7 as done here) through the use of the variational equations.
  This involves solving a non-autonomous first order linear differential equation with coefficients dependent 
  on the trajectories on the collision manifold and the coefficients of the vector field near the collision manifold. In the planar problem 
  the trajectories on the collision manifold are significantly more complicated
  and worse so, the coefficients in the vector field as computed from the normal form are far too unwieldy. 
  This is why in this paper the error terms are one order less than in the collinear paper \cite{duignanC83regularisationSimultaneous2020}.

  The fact that $F_j$ is polynomial in $\Gamma_1^*,\Gamma_2^*$ and, due to the
  fact $\pib$ is $ C^{8/3}$, $F_j$ is not identically zero, it is guaranteed that
  the set of exceptional points is of measure 0. The interesting question of
  whether this set is empty remains. In particular, there are two questions
  worth pursuing:
  \begin{enumerate}
  \item Is there a choice of asymptotic point $\Gamma_1^*,\Gamma_2^*$ for which
    $F_j=0$?
  \item Is there a choice of ${\ba}_0,{\ga}_0,{\da}_0$ and a choice of
    $\Gamma_1^*,\Gamma_2^*$ for which $F_j=0$?
  \end{enumerate}
  Of course question 2 is a weaker question than 1. An answer to either would
  provide regions of phase space for which the simultaneous binary collision is
  more regularisable, perhaps even $C^\infty$. In such a case, by finding
  sub-problems, like the trapezoidal, caledonian, collinear, etc, which are
  contained within these regions, one may be able to find a Levi-Civita type
  transformation for these sub-problems which smoothly regularise the collision.

	\bibliography{SBC_Planar}
	\bibliographystyle{hplain}	
	
	\appendix
	\section{Potential and Normal Form Functions}\label{sec:Constants}

The mass constants in \eqref{eqn:potential} are given by
% \begin{equation}
%   \label{eq:potentialconstants}
%   \begin{aligned}
%     b_0 	&= d_1 + d_2 + d_3 + d_4 = (m_1 + m_2)(m_3 + m_4), \\ \displaystyle
%     b_{12}	&= \frac{1}{8}\left( C_1^2 (d_3 + d_4) + C_2^2 (d_1 + d_2)\right), \qquad b_{22}	= \frac{1}{8}\left(  C_4^2 (d_1 + d_3) + C_3^2 (d_2 + d_4)\right), \\ \displaystyle
%     b_{13}	&= \frac{1}{16}\left(C_1^3 (d_3 + d_4) - C_2^3 (d_1 + d_2)\right), \qquad b_{23}	= \frac{1}{16}\left(C_4^3 (d_1 + d_3) - C_3^3 (d_2 + d_4)\right), \\ \displaystyle
%     b_{14}	&= \frac{1}{128}\left(C_1^4 (d_3 + d_4) + C_2^4 (d_1 + d_2)\right), \qquad b_{24}	= \frac{1}{128}\left(C_4^4 (d_1 + d_3) + C_3^4 (d_2 + d_4)\right), \\ \displaystyle
%     b_c		&= \frac{3}{64}( C_1^2 (C_4^2 d_3 + C_3^2 d_4) + C_2^2 (C_4^2 d_1 + C_3^2 d_2) ), \\
%     C_1 &= a_1^{1/3} \frac{1}{8 k_1 M_1} c_1, \quad C_2 = a_1^{1/3} \frac{1}{8
%       k_1 M_1} c_2, \quad C_3 = a_2^{1/3} \frac{1}{8 k_2 M_2} c_3, \quad C_4 =
%     a_2^{1/3} \frac{1}{8 k_2 M_2} c_4.
%   \end{aligned}
% \end{equation}
\begin{equation}\label{eq:potentialconstants}
\begin{array}{rclrcl}
    b_0  &=& (m_1 + m_2)(m_3 + m_4), \\
    b_{12} &=& b_0 \frac{m_1 m_2}{8 (m_1 + m_2)^2}, &
    b_{22} &=& b_0 \frac{m_3 m_4}{8 (m_3 + m_4)^2}, \\
    b_{13} &=& b_0 \frac{m_1 m_2( m_1 - m_2)}{16(m_1 + m_2)^3}, &
    b_{23} &=& b_0 \frac{m_3 m_4(m_3 - m_4)}{16(m_3 + m_4)^3}, \\
    b_{14} &=& b_0 \frac{m_1 m_2(m_1^2 - m_1 m_2 + m_2^2)}{128(m_1 + m_2)^4}, &
    b_{24} &=& b_0 \frac{m_3 m_4(m_3^2 - m_3 m_4 + m_4^2)}{128(m_3 + m_4)^4}, \\
    b_c &=& \frac{3}{64} M_1 M_2
\end{array}
\end{equation}

Each of the homogeneous polynomials $W_j$ are given by,
\begin{equation}
  \label{eq:hompolys}
  \begin{aligned}
    W_2(Q) &= 3 \bar{Q}^2+2\bar{Q} Q + 3 Q^2 \\
    W_3(Q) &= 5\bar{Q}^3 + 3 \bar{Q}^2 Q + 3 \bar{Q}Q^2 + 5 Q^3 = (Q + \bar{Q})(5 Q^2 - 2 Q \bar{Q} + 5 \bar{Q}^2)\\
    W_4(Q) &= 35 \bar{Q}^4+20 \bar{Q}^3Q+18 \bar{Q}^2 Q^2+ 20 \bar{Q} Q^3 + 35 Q^4 \\
    W_c(Q_1,Q_2) &= 35 \bar{Q}_1^2 \bar{Q}_2^2 + 10 \bar{Q}_1 \bar{Q}_2^2 Q_1 +
    3 \bar{Q}_2^2 Q_1^2 + 10 \bar{Q}_1^2 \bar{Q}_2 Q_2 + 12 \bar{Q}_1 \bar{Q}_2
    Q_1 Q_2 \\
    &\qquad + 10 \bar{Q}_2 Q_1^2 Q_2 + 3 \bar{Q}_1^2 Q_2^2 + 10 \bar{Q}_1 Q_1
    Q_2^2 + 35 Q_1^2 Q_2^2
  \end{aligned}
\end{equation}

		The normal form functions are given by
		\begin{equation}
		\begin{aligned}
      R_1^{4}(\zeta_1,\zeta_2, h_1, h_2) &= \tfrac{4}{5}L_2^2 \left( (h_1+h_1^*)L_1^2 - (h_2+h_2^*)L_2^2 \right) \\
      R_1^{6}(\zeta_1,\zeta_2,h_1,h_2) &= -\tfrac{16}{25}L_1^2L_2^4(h_1+h_1^*)(h_2+h_2^*) - \tfrac{4}{107925} (h_1+h_1^*)^2 R_{11}^6 + \tfrac{4}{107925}(h_2+h_2^*) R_{12}^6 \\
      R_{11}^{6}(\zeta_1,\zeta_2,h_1,h_2) &= 33300 I_1^4 I_2^2 - 21645 I_1 I_2^5
      + 34965 I_1^4 L_2^2 + 11285 I_1^2 L_1^2 L_2^2 \\
      &\qquad - 23026 L_1^4 L_2^2 - 59385 I_1 I_2^3 L_2^2 - 46990 I_1 I_2 L_2^4 \\
      R_{12}^{6}(\zeta_1,\zeta_2,h_1,h_2) &= 6105 I_1^6 + 34965 I_1^4 L_1^2 +
      46990 I_1^2 L_1^4 + 18130 L_1^6\\
      &\qquad + 5550 I_1^3 I_2^3 - 10545 I_1 L_1^2 I_2^3 - 5550 I_2^6 - 11285
      I_1 L_1^2 I_2 L_2^2 + 78972 L_2^6
      % R_1^{6}(\zeta_1,\zeta_2,h_1,h_2)	&= \frac{64}{1375950790}(h_1 + h_{1}^*)^2 R_{11}^{6}(\zeta_1,\zeta_2) + \frac{37}{1375950790}(h_2 + h_{2}^*)^2 R_{12}^{6}(\zeta_1,\zeta_2) \\
      % R_2^{6}(\zeta_1,\zeta_2,h_1,h_2)	&= L_1 \left( \frac{2066524}{687975395}(h_1 + h_{1}^*)^2 R_{21}^{6}(\zeta_1,\zeta_2) + \frac{148}{687975395} (h_2 + h_{2}^*)^2 R_{22}^{6}(\zeta_1,\zeta_2) \right) \\
      % R_{11}^{6}(\zeta_1,\zeta_2) &= I_1\left( -22238428 I_1^3 I_2^2 + 3183564
      %   I_1 L_1^2 I_2^2 + 14805827 I_2^5 - 17459204 I_1^3 L_2^2 \right. \\
      % &\qquad \left. - 3183564 I_1 L_1^2 L_2^2 + 25955974 I_2^3 L_2^2 + 10511803 I_2 L_2^4 \right) \\
      % R_{12}^{6}(\zeta_1,\zeta_2) &= 3007511 I_1^6 + 9043997 I_1^4 L_1^2 +
      % 6973181 I_1^2 L_1^4 + 936695 L_1^6 + 29178160 I_1^3 I_2^3 \\
      % &\qquad + 21999968 I_1 L_1^2 I_2^3 - 13232544 I_2^6 + 6367128 I_1^3 I_2
      % L_2^2 - 6367128 I_1 L_1^2 I_2 L_2^2 \\
      % R_{21}^{6}(\zeta_1,\zeta_2) &= 11 I_1^3 I_2^2 + 11 I_1 L_1^2 I_2^2 + 5
      % I_2^5 - 11 I_1^3 L_2^2 - 11 I_1 L_1^2 L_2^2 - 2 I_2^3 L_2^2 - 7 I_2 L_2^4 \\
      % R_{22}^{6}(\zeta_1,\zeta_2) &= 905597 I_1^5 + 2152522 I_1^3 L_1^2 +
      % 1246925 I_1 L_1^4 - 2549867 L_1^2 I_2^3 \\
      % &\qquad + 765412 L_1^2 I_2L_2^2 + -1080001 I_1^2I_2^3 + 704454 I_1^2I_2
      % L_2^2
		\end{aligned}
		\end{equation}

Writing $\alpha_i = 4 \arg( \Gamma_i)$ the first few terms in $R_5^9$ are
\begin{equation} \label{eqn:R59}
    R_5^9 = \frac{8}{19} (I_1 - I_2) ( I_1^2 + I_1 I_2 + I_2^2) ( I_1^6 - 11 I_1^3I_2^3 + I_2^6) 
     ( 6 + 10 \cos \alpha_1 + 10 \cos \alpha_2 + 35 \cos (\alpha_1 + \alpha_2) + 3 \cos (\alpha_1 - \alpha_2)) +
\end{equation}
\[
     - \frac{2}{209} I_1 L_1^2 (5 I_1^6 - 35 I_1^3 I_2^3 + 14 I_2^6)
      ( 109 (10 \cos\alpha_1 + 35 \cos(\alpha_1 + \alpha_2) + 3 \cos(\alpha_1 - \alpha_2)) - 86 (3 + 5 \cos\alpha_2 ) )
      + \{ 1 \leftrightarrow 2 \} +
\]
\[ 
   +  \frac{256}{13} L_1^3( I_1^6 - 5 I_1^3 I_2^3 + I_2^6 )L_1^3 (10 \sin\alpha_1 + 35 \sin(\alpha_1 + \alpha_2) + 3 \sin(\alpha_1 - \alpha_2)  - \{ 1 \leftrightarrow 2 \} )
\]
plus higher order terms in $L_i$, 
where $ \{ 1 \leftrightarrow 2 \} $ denotes the exchange of indices in the preceding term. 
    
    The relevant functions for the integral $\kappa$ in equation
    \eqref{eqn:kappaOriginal} are given by,
    \begin{equation}
      \label{eq:kappaFunctions}
      \begin{aligned}
        G^5(\zeta_1,\zeta_2,h_1,h_2) &= \-\tfrac45(I_1^3 + 6 I_1 L_1^2 - I_2^3)((h_1 + h_1^*)L_1^2 - (h_2+h_2^*)L_2^2)\\
        G^7(\zeta_1,\zeta_2,h_1,h_2) &= (h_1+h_1^*)^2
        \left(  G_1^7(\zeta_1,\zeta_2) +\tfrac12 G_2^7(\zeta_1,\zeta_2)\right) +(h_1+h_1^*)(h_2+h_2^*) G_3^7(\zeta_1,\zeta_2)\\
        &\qquad - (h_2+h_2^*)^2\left(  G_1^7(\zeta_2,\zeta_1) -\tfrac12 G_2^7(\zeta_2,\zeta_1)\right)  \\
        G_1^7(\zeta_1,\zeta_2) &= -2937060 I_1^4 L_1^2 I_2 - 3947160 I_1^2 L_1^4
        I_2
        - 1522920 L_1^6 I_2 + 1107225 I_1^3 I_2^4 + 3181815 I_1 L_1^2 I_2^4 \\
        &\qquad - 807525 I_2^7 + 2447550 I_1^3 I_2^2 L_2^2 + 6394710 I_1 L_1^2
        I_2^2 L_2^2 -
        2692305 I_2^5 L_2^2 \\
        &\qquad - 944468 I_1^3 L_2^4 - 899220 I_1 L_1^2 L_2^4 - 1503082 I_2^3
        L_2^4 - 3800244 I_2 L_2^6 \\
        G_2^7 &= 56 (18315 I_1^6 I_2 - 27973 I_1^3 L_2^4 - 32115 I_1 L_1^2 L_2^4
        +
        27973 I_2^3 L_2^4 + 135723 I_2 L_2^6) \\
        G_3^7(\zeta_1,\zeta_2) &= -\tfrac{16}{25}L_1^2 (I_1^3 + 12 I_1 L_1^2 -
        I_2^3) L_2^2
      \end{aligned}
    \end{equation}
		
	\section{Proof of Lemma \ref{lem:ApproxDulacMap}} \label{sec:proofOfLemma}
		In this section we state and use the required technicalities from \cite{duignanNormalFormsManifolds} to prove Lemma \ref{lem:ApproxDulacMap}. The theory developed in \cite{duignanNormalFormsManifolds} is more general than what is needed in this appendix. As such, all referenced propositions have been modified for their particular application in this work.
		
		Assume that the one-dimensional invariant manifold leaving a point on $ \NHIM $ in the normal direction is stable. By Proposition \ref{prop:structureOfCollisionManifold}, the non-zero eigenvalues of $ \NHIM $ are $ (-1,3,1,1) $ up to scaling. Assume that coordinates $ (x,y,z,w,u)\in\R^4\times\R^8 $ have been chosen near $ \NHIM $ so that the local stable manifold is given by $ y=z=w=0 $ and the local unstable manifold by $ x = 0 $. Due to the fact the non-zero eigenvalues of $ \NHIM $ are resonant, it is not, in general, possible to linearise the vector field in a tubular neighbourhood of $ \NHIM $. The following proposition from \cite{duignanNormalFormsManifolds} gives the `simplest' representation of a vector field in a neighbourhood of $ \NHIM $. This is the so called normal form.
		
		\begin{proposition}[\cite{duignanNormalFormsManifolds}]\label{prop:normalFormNHIM}
			Define the monomials, 
			\[ U_y = x^3 y,\quad U_z = x z,\quad U_w = x w \]
			and let  $ n = (n_1,n_2,n_3)\in\N^3 $ and $ |n|:=n_1+n_2+n_3 $. There exists a smooth, near-identity transformation $ \Phi $ and a smooth time rescaling bringing $ X $ into the normal form $ X_N $,
			\begin{equation}\label{eqn:generalNormalForm}
				\begin{aligned}
				\dot{x}		&= -x \\
				\dot{y}		&=  3 y + y \sum_{|n| \geq 1} G_y^{(n)}(u) U_y^{n_1} U_z^{n_2} U_w^{n_3} + y U_y^{-1}  \sum_{n_1+n_2 \geq 3} B_y^{(n)}(u) U_z^{n_1} U_w^{n_2} \\ 
				\dot{z}		&= z + z \sum_{|n| \geq 1} G_z^{(n)}(u) U_y^{n_1} U_z^{n_2} U_w^{n_3} + z U_z^{-1} \sum_{n_1+n_2 \geq 1} B_z^{(n)}(u) U_y^{n_1} U_w^{n_2} \\ 
				\dot{w}		&= w + w \sum_{|n| \geq 1}  G_w^{(n)}(u) U_y^{n_1} U_z^{n_2} U_w^{n_3} + w U_w^{-1} \sum_{n_1+n_2 \geq 1} B_w^{(n)}(u) U_y^{n_1} U_z^{n_2} \\
				\dot{u}		&= \sum_{|n| \geq 1}  G_u^{(n)}(u)  U_y^{n_1} U_z^{n_2} U_w^{n_3}, \\
				\end{aligned}
			\end{equation}
			where $ G^{(n)}_\eta(u), B^{(n)}_\eta(u) $ are smooth functions in $ u $ for each $ \eta = x,y,z,w,u $.
		\end{proposition}
		
		Assume that $ (x,y,z,w,u) $ are the coordinates so that the vector field in a tubular neighbourhood of $ \NHIM $ is given by the normal form \eqref{eqn:generalNormalForm}. Consider the section $ \Sigma = [0,1] \times[-1,1]^3\times \R^8 $ defined in the normal form coordinates and its various faces,
		\[ \Sigma_x := \Sigma\cap \{x=1 \},\qquad \Sigma_\eta^\pm := \Sigma\cap\{\eta=\pm 1 \},,\quad \eta = y,z,w. \] 
		Denote the \emph{Dulac map} by the continuous map,
		\[ D: \Sigma_y^\pm \cup\Sigma_z^\pm\cup\Sigma_w^\pm  \to \Sigma_x^\pm, \]
		and its action restricted to each face by,
		\[ D_\eta^\pm : \Sigma_{\eta}^\pm \to \Sigma_x,\quad \eta = y,z,w. \]
		
		The following proposition, adapted from \cite{duignanNormalFormsManifolds}, gives the asymptotic properties of $ D $.
		
		\begin{proposition}[\cite{duignanNormalFormsManifolds}]\label{prop:DirectionalDulacMaps}
			Let $ (y_0,z_0,w_0,u_0), (x_1,y_1,z_1,w_1,u_1) $ be coordinates on $ \Sigma_x, \Sigma_{\eta},\eta = y,z,w $ respectively, and define 
			\[ U_y^0 = x_1^3 y_1,\quad U_z^0 = x_1 z_1,\quad U_w^0 = x_1 w_1,\quad U_0^{n} := \left(U_y^0\right)^{n_1} \left(U_z^0\right)^{n_2}\left(U_y^0\right)^{n_3} \]
			for $ n = (n_1,n_2,n_3)\in\N^3 $. Then the components of the Dulac map have the asymptotic series,
			\begin{equation}\label{eqn:DulacMapsForward}
				\begin{aligned}
					y_0	&\sim x_1^3 \left( y_1 + y_1 \sum_{|n| \geq 1} \bar{G}_y^{(n)}(u_1,\ln x_1) U_0^n + x_1^{-3}\sum_{n_1+n_2\geq 3} \bar{B}_y^{(n_1,n_2)}(u_1,\ln x_1) \left(U_z^0\right)^{n_1} \left(U_w^0\right)^{n_2}  \right) \\
					z_0 &\sim x_1 \left( z_1 + z_1 \sum_{|n| \geq 1} \bar{G}_z^{(n)}(u_1,\ln x_1) U_0^n + x_1^{-1}\sum_{n_1+n_2\geq 1} \bar{B}_z^{(n_1,n_2)}(u_1,\ln x_1) \left(U_y^0\right)^{n_1} \left(U_w^0\right)^{n_2}  \right) \\
					w_0 &\sim x_1 \left( w_1 + w_1 \sum_{|n| \geq 1} \bar{G}_w^{(n)}(u_1,\ln x_1) U_0^n + x_1^{-1}\sum_{n_1+n_2\geq 1} \bar{B}_w^{(n_1,n_2)}(u_1,\ln x_1) \left(U_y^0\right)^{n_1} \left(U_z^0\right)^{n_2}  \right)  \\
					u_0	&\sim u_1 + \sum_{|n| \geq 1} \bar{G}_u^{(n)}(u_1,\ln x_1) U_2^{n}
				\end{aligned}
			\end{equation}	
			with $ y_1,z_1,w_1 $ set to $ \pm 1 $ for $ D_y^\pm,D_z^\pm,D_w^\pm $ respectively. Each coefficient function $ K^{(n)} = \bar{G}^{(n)},\bar{B}^{(n)} $ are:
			\begin{enumerate}[(i)]
				\item Polynomial in $ \ln x_1 $ with vanishing constant term and smooth in $ u_1 $. 
				\item Polynomial in $ B_\eta^{(\tilde{n})}(u_1),G_\eta^{(\tilde{n})}(u_1), \eta = y,z,w,u, $ with vanishing constant term for all $ |\tilde{n}| \leq |n| $.
			\end{enumerate}
			Moreover, $ B^{(n)}=0 $ (resp. $ G^{(n)}_u = 0 $) for $ |n| \leq m $ if and only if $ \bar{B}^{(n)}=0 $ (resp. $ G^{(n)}_u = 0 $) for $ |n| \leq m $.
		\end{proposition}
	
		A crucial point in Proposition \ref{prop:DirectionalDulacMaps} is the fact that, if  $ B_\eta^{(\tilde{n})}(u_0),G_\eta^{(\tilde{n})}(u_0), \eta = y,z,w,u, $ vanishes for all $ |\tilde{n}| \leq |n| $, then the coefficients, $ \bar{G}^{(\tilde{n})},\bar{B}^{(\tilde{n})} $, in the Dulac maps vanish as well. If one only knows the normal form \eqref{eqn:generalNormalForm} up to some order in $ (x,y,z,w) $ then one can infer to what order $ D $ is known. For example, if it is known that the normal form $ X_N $ in \eqref{eqn:generalNormalForm} has no resonant terms of type $ G_\eta^{(n)}(u) $ for $ |n| < m $ and $ B_\eta^{(n)}(u) $ for $ n_1 + n_2 < m + 1 $ and $ \eta = y,z,w,u $, then the Dulac maps are known to $ |n| = m $ in the $ \bar{G} $ summation and $ n_1 + n_2 = m+1 $ in the $ \bar{B} $ summation.
		
		To prove Lemma \ref{lem:ApproxDulacMap} an understanding of the inverse map $ D^{-1} $ is also required. The inverse is not continuous and as such, does not admit an asymptotic expansion in the usual sense. However, by appropriately rescaling $ (y_0,z_0,w_0) = \e(y_0,z_0,w_0) $, with $ \e \ll 1 $, one can obtain compositional inverse of $ D $ using $ \e $ as the small parameter. The asymptotic structure of $ D^{-1} $ is not, in general, as nice as the forward map $ D $. However, if the resonant terms are known to vanish to some order, then one can quickly compute the inverse directional Dulac maps up to this order. We prove this in the following lemma.
		
		\begin{lemma}\label{prop:invDOrder}
			Let $ m_1,m_2 > 0 $ and suppose that  $ G_\eta^{(n)}=0 $ for $ |n| < m_1,  \eta = x,y,z,w $, that $ B_\eta^{(n_1,n_2)} = 0 $ for $ n_1+n_2 < m_1+1, \eta = z,w $, and that $ G_y^{(n)}=0 $ for $ n_1+n_2 < m_1+3 $. Moreover, assume that $ G_u^{(n)}=0 $ for $ |n| <m_2 $.
			
			If $ (\e y_0,\e z_0,\e w_0,\e u_0) $ are scaled coordinates on $ \Sigma_0 $, then each of the directional Dulac maps have inverses asymptotic to,
			\begin{equation}\label{eqn:invDirectionalDulacMap}
			\begin{aligned}
				D_{1,y} &:\quad
				\begin{aligned}
					x_1	&\sim \e^{1/3} y_0^{1/3} \left(1 + O(\e^{m_1}\ln\e) \right) \\
					z_1	&\sim \e^{2/3} y_0^{-1/3} z_0\left(1 + O(\e^{m_1}\ln\e)\right)
				\end{aligned}\quad 
				\begin{aligned}
					w_1	&\sim \e^{2/3} y_0^{-1/3} w_0 \left(1 + O(\e^{m_1}\ln\e) \right)\\
					u_1	&\sim  u + O(\e^{m_2}\ln\e)
				\end{aligned} \\[1em]
				\vspace*{10pt}
				D_{1,z} &:\quad
				\begin{aligned}
					x_1	&\sim \e z_0 \left(1 + O(\e^{m_1}\ln\e) \right) \\
					y_1	&\sim \e^{-2} z_0^{-3} y_0 \left(1 + O(\e^{m_1}\ln\e)\right)
				\end{aligned}\quad 
				\begin{aligned}
					w_1	&\sim z_0^{-1} w_0 \left(1 + O(\e^{m_1}\ln\e) \right)\\
					u_1	&\sim  u + O(\e^{m_2}\ln\e)
				\end{aligned} \\[1em]
				D_{1,w} &:\quad
				\begin{aligned}
					x_1	&\sim \e w_0 \left(1 + O(\e^{m_1}\ln\e) \right) \\
					y_1	&\sim \e^{-2} w_0^{-3} y_0 \left(1 + O(\e^{m_1}\ln\e)\right)
				\end{aligned}\quad 
				\begin{aligned}
					z_1	&\sim w_0^{-1} z_0 \left(1 + O(\e^{m_1}\ln\e) \right)\\
					u_1	&\sim u + O(\e^{m_2}\ln\e)
				\end{aligned}
				\end{aligned}
			\end{equation}
		\end{lemma}
		\begin{proof}
			We will prove the lemma for $ D_y^+ $ only as the other cases follow analogously. By hypothesis we have $ G_\eta^{(n)}=0 $ for $ |n| < m_1,  \eta = x,y,z,w $, that $ B_\eta^{(n_1,n_2)} = 0 $ for $ n_1+n_2 < m_1+1, \eta = z,w $, and that $ G_y^{(n)}=0 $ for $ n_1+n_2 < m_1+3 $ in the normal form \eqref{eqn:generalNormalForm}. From Proposition \ref{prop:DirectionalDulacMaps}, particular the remark after `moreover', the map $ D_y $ is asymptotic to,
			\begin{equation}\label{eqn:Dyasymp}
				\begin{aligned}
					y_0	&\sim x_1^3 \left( 1 + \sum_{|n| \geq m_1 } \bar{G}_y^{(n)}(u_1,\ln x_1) U_0^n + x_1^{-3}\sum_{n_1+n_2\geq 3 + m_1} \bar{B}_y^{(n_1,n_2)}(u_1,\ln x_1) \left(U_z^0\right)^{n_1} \left(U_w^0\right)^{n_2}  \right) \\
					z_0 &\sim x_1 \left( z_1 + z_1 \sum_{|n| \geq m_1 } \bar{G}_z^{(n)}(u_1,\ln x_1) U_0^n + x_1^{-1}\sum_{n_1+n_2\geq 1 + m_1} \bar{B}_z^{(n_1,n_2)}(u_1,\ln x_1) \left(U_y^0\right)^{n_1} \left(U_w^0\right)^{n_2}  \right) \\
					w_0 &\sim x_1 \left( w_1 + w_1 \sum_{|n| \geq m_1} \bar{G}_w^{(n)}(u_1,\ln x_1) U_0^n + x_1^{-1}\sum_{n_1+n_2\geq 1 + m_1} \bar{B}_w^{(n_1,n_2)}(u_1,\ln x_1) \left(U_y^0\right)^{n_1} \left(U_z^0\right)^{n_2}  \right)  \\
					u_0	&\sim u_1 + \sum_{|n| \geq m_2} \bar{G}_u^{(n)}(u_1,\ln x_1) U_2^{n}
				\end{aligned}
			\end{equation}
			Taking the scaled coordinates $ (\e y_0,\e z_0,\e w_0,\e u_0) $ to keep track of relative size, then this form of $ D_y^+ $ gives the first order asymptotics,
			\[ x_1 \sim \e^{1/3} y_0^{1/3},\quad z_1 \sim \e^{2/3} z_0 y_0^{-1/3},\quad  w_1 \sim \e^{2/3} w_0 y_0^{-1/3},\quad u_1 \sim u_0. \]
			it follows that each of the monomials $ U_y^0,U_z^0,U_w^0 $ hs the leading order asymptotics, 
			\[ U_y^0 = \e x_1^3 ,\quad U_z^0 = \e  x_1 z_1,\quad U_w^0 = \e x_1 w_1,\quad \]
			Substituting these two leading order asymptotics into equation \eqref{eqn:Dyasymp} concludes the lemma.
		\end{proof}
		
		At last Lemma \ref{lem:ApproxDulacMap} can be proved.
		\begin{proof}[Proof of Lemma \ref{lem:ApproxDulacMap}]
			The lemma follows provided we can show the hypothesis of Lemma \ref{prop:invDOrder} is satisfied for $ m_1 = 2, m_2=3 $. That is, we need to show $ G_\eta^{(n)}=0 $ for $ |n| < 2,  \eta = \ra,\ba,\tga,\da $, that $ B_\eta^{(n_1,n_2)} = 0 $ for $ n_1+n_2 < 3, \eta = \,w $, that $ G_{\tga}^{(n)}=0 $ for $ n_1+n_2 < 5 $, and finally that $ G_u^{(n)}=0 $ for $ |n| < 3 $.
			
			We show only the first hypothesis, $ G_\eta^{(n)}=0 $ for $ |n| < 2 $, as the others follows analogously. Observe that $ U_0^n = \ra^{3 n_1+n_2+n_3}\tga^{n_1}\ba^{n_2}\da^{n_3}  $. Now, system \eqref{eqn:ApproxNF} has no resonant terms with $ \ra \leq 4 $. It follows that all resonant terms $ G_\eta^{(n)} U_0^n  $ in the normal form \ref{eqn:generalNormalForm} for $ |n| = 1 $ must have $ G_\eta^{(n)} = 0 $ as desired.
		\end{proof}

\end{document}